\documentclass[11pt]{article}
\usepackage{amssymb,amsmath,amsthm,amsfonts,verbatim}
\usepackage[all,cmtip]{xy}
\usepackage{hyperref}
\usepackage[top=1.2in,bottom=1.2in,left=1.21in,right=1.21in]{geometry}
\usepackage{todonotes}
\usepackage{qtree}

\usepackage{extarrows}
\usepackage[all,cmtip]{xy}

\usepackage{tikz-cd}
\tikzset{
	commutative diagrams/.cd, 
	arrow style=tikz, 
	diagrams={>=stealth}
}

\newtheorem{thm}{Theorem}[section]
\newtheorem{prop}[thm]{Proposition}
\newtheorem{cor}[thm]{Corollary}
\newtheorem{lem}[thm]{Lemma}
\newtheorem{fact}[thm]{Fact}

\newtheorem{question}[thm]{Question}
\newtheorem{claim}{Claim}

\newtheorem{thmx}{Theorem}

\newtheorem{definition}[thm]{Definition}

\theoremstyle{definition}

\newtheorem{example}[thm]{Example}
\newtheorem{problem}[thm]{Problem}

\numberwithin{equation}{section}

\theoremstyle{remark}
\newtheorem{remark}[thm]{Remark}
\newtheorem{note}[thm]{Note}

\newcommand{\R}{\mathbb{R}}  
\newcommand{\C}{\mathbb{C}}  
\newcommand{\Proj}{\mathbb{P}}  
\newcommand{\N}{\mathbb{N}}  
\newcommand{\Z}{\mathbb{Z}}  
\newcommand{\Q}{\mathbb{Q}}  
\newcommand{\F}{\mathbb{F}}  

\newcommand{\squareDiagram}[8]{\begin{equation*}
			\begin{tikzcd}[sep=2.5em, ampersand replacement=\&]
				{#1} \arrow[r, "{#5}"] \arrow[d, swap, "{#6}"] \& {#2} \arrow[d, "{#7}"] \\ 
				{#3} \arrow[r, swap, "{#8}"] \& {#4}
			\end{tikzcd}
		\end{equation*}}

\newcommand{\mor}[1]{\ensuremath{\overset{#1}{\longrightarrow}}} 

\newcommand{\topoplus}[1][{}]{\mathbin{\underset{\substack{#1}}{\oplus}}}
\newcommand{\topotimes}[1][{}]{\mathbin{\underset{\substack{#1}}{\otimes}}}
\newcommand{\PO}[2]{\operatorname{PO}\binom{#1}{#2}}

\DeclareMathOperator{\Ho}{H} 
\DeclareMathOperator{\Hom}{Hom} 
\DeclareMathOperator{\op}{op} 
\DeclareMathOperator{\Tr}{Tr} 
\DeclareMathOperator{\Ind}{Ind} 
\DeclareMathOperator{\Res}{Res} 
\DeclareMathOperator{\Aut}{Aut} 

\DeclareMathOperator{\Rat}{Rat} 
\DeclareMathOperator{\PConf}{PConf} 
\DeclareMathOperator{\PRat}{PRat} 
\DeclareMathOperator{\Char}{Char} 
\DeclareMathOperator{\ev}{ev} 

\newcommand{\catname}[1]{{\normalfont\operatorname{#1}}}
\newcommand{\FI}{\catname{FI}}
\newcommand{\VI}{\catname{VI}}

\newcommand{\Set}{\catname{Set}}
\newcommand{\CCat}{{\operatorname{\textbf{C}}}}

\newcommand{\Vect}{\catname{Vect}}
\newcommand{\Mod}[1]{\catname{Mod}_{#1}}

\renewcommand{\emptyset}{\varnothing}

\begin{document}


\title{Categories of FI type: a unified approach to generalizing representation stability and character polynomials}


\author{Nir Gadish}





\maketitle

\begin{abstract}
Representation stability is a theory describing a way in which a sequence of representations of different groups is related, and essentially contains a finite amount of information. Starting with Church-Ellenberg-Farb's theory of $\FI$-modules describing sequences of representations of the symmetric groups, we now have good theories for describing representations of other collections of groups such as finite general linear groups, classical Weyl groups, and Wreath products $S_n\wr G$ for a fixed finite group $G$. This paper attempts to uncover the mechanism that makes the various examples work, and offers an axiomatic approach that generates the essentials of such a theory: character polynomials and free modules that exhibit stabilization.

We give sufficient conditions on a category $\CCat$ to admit such structure via the notion of categories of $\FI$ type. This class of categories includes the examples listed above, and extends further to new types of categories such as the categorical power $\FI^m$, whose modules encode sequences of representations of $m$-fold products of symmetric groups. The theory is applied in \cite{Ga} to give homological and arithmetic stability theorems for various moduli spaces, e.g. the moduli space of degree $n$ rational maps $\Proj^1\rightarrow \Proj^m$.
\end{abstract}




\section{Introduction}
The purpose of this paper is to describe a categorical structure that is responsible for the existence of representation stability phenomena. Our approach is centered around free modules\footnote{These are commonly called \#-modules in the context of the category $\FI$.} and character polynomials (defined below). We show that our proposed categorical structure gives rise to free modules which satisfy the fundamental properties that produce representation stability, and in particular the Noetherian property. We take an axiomatic approach that applies in a broad context, generalizing many of the known examples.

\subsection{Motivation}
Let $\FI$ be the category of finite sets and injections. An \emph{$\FI$-module} is a functor from $\FI$ to the category of modules over some fixed ring $R$. An $\FI$-module $M_\bullet$ is a single object that packages together a sequence of representations of the symmetric groups $S_n$ for every $n\in \N$ (see e.g. \cite{CEF}). Objects of this form arise naturally in topology and representation theory, for example:
\begin{itemize}
\item Cohomology of configuration spaces $\{\PConf^n(X)\}_{n\in \N}$ for a manifold $X$.
\item Diagonal coinvariant algebras $\{ \Q[x_1,\ldots,x_n,y_1,\ldots,y_n]/\mathcal{I}_n \}_{n\in N}$ (see \cite{CEF}).
\end{itemize}

A fundamental result of Church-Ellenberg-Farb \cite{CEF} is that an $\FI$-module over $\Q$ is finitely-generated, i.e. there exists a finite set of elements not contained in any proper submodule, if and only if the sequence of $S_n$-representations stabilizes in a precise sense (see \cite{CEF} for details). This phenomenon was named \emph{representation stability}. In particular, if one defines class functions
$$
X_k(\sigma) = \# \text{ of $k$-cycles in } \sigma
$$
simultaneously on all $S_n$, then \cite{CEF} show that for every finitely-generated $\FI$-module $M_\bullet$ then there exists a single polynomial $P\in \Q[X_1,X_2,\ldots]$ -- a \emph{character polynomial} -- that describes the characters of the $S_n$-representations $M_n$ independent of $n$ for all $n\gg1$.

The uniform description of the characters in terms of a single character polynomial accounts for the most direct applications of the theory, for example:
\begin{itemize}
\item For every manifold $X$ and $i\geq 0$, the dimensions of $\{\Ho^i(\PConf^n(X);\Q)\}_{n\in \N}$ are given by a single polynomial in $n$ for all $n\gg1$.\footnote{See \cite{Church-manifolds}.}
\item Every polynomial statistic, regarding the irreducible decomposition of degree-$n$ polynomials over $\F_q$, tends to an asymptotic limit as $n\rightarrow\infty$.\footnote{See \cite{CEF-pointcounts}.}
\end{itemize}
However, the above logic could be reversed: as first suggested by Gan-Li in \cite{GL-coinduction}, Nagpal showed in \cite[Theorem A]{Na} that if $M_\bullet$ is a finitely generated $\FI$-module, then in some range $n\gg 1$ it admits a finite resolution by free $\FI$-modules (see below) and these have characters given by character polynomials. It follows that for every $n \gg 1$ the character of the $M_{n}$ is itself given by a character polynomial. One can then get stabilization of the decomposition of $M_n$ into irreducible representations as a corollary of this fact! We assert that the key property of character polynomials -- responsible for all representation stability phenomena and applications -- is the following.
\begin{fact}[{\cite[Theorem 3.9]{CEF-pointcounts}}]\label{fact:stabilization}
The $S_n$-inner product of two character polynomials $P$ and $Q$ becomes independent of $n$ for all $n\geq \deg(P)+\deg(Q)$.
\end{fact}

The benefit of Gan-Li's and Nagpal's approach is that free $\FI$-modules and character polynomials readily generalize to a wide class of categories similar to $\FI$, and do not require any understanding of the representation theory of the individual automorphism groups. Thus representation stability extends whenever these structures exist.

\subsection{Generalization to other categories}
Work on generalizing representation stability to other contexts has proceeded in several partially overlapping directions. A major direction on which we will be focused is that of modules over other categories $\CCat$ of injections, whose automorphism groups are of interest. Let $\CCat$ be a category.
\begin{definition}[\textbf{$\CCat$-modules}]\label{def:intro_modules}
A \emph{$\CCat$-module} over a ring $R$ is a covariant functor 
$$
M_\bullet: \CCat \mor{}\Mod{R}.
$$
For every object $c$, the evaluation $M_c$ is naturally a representation of the group $\Aut_\CCat(c)$ in $R$-modules, and these representation are related by the morphisms of $\CCat$.
\end{definition}
One then studies this category of representations,  describes the simultaneous class functions that generalize character polynomials, and proves the analog of Fact \ref{fact:stabilization}. For example:
\begin{enumerate}
\item Putman-Sam \cite{PS} considered the category $\CCat=\VI_q$ of finite dimensional vector spaces over $\F_q$ and injective linear maps, whose representations encode sequences of $\operatorname{Gl}_n(\F_q)$-representations.
\item Wilson \cite{Wi} studied $\CCat=\FI_{\mathcal{W}}$ whose automorphism groups are the classical Weyl groups $\mathcal{W}_n$ of type $B/C$ or $D$.
\item Sam-Snowden \cite{SS-FIG} and Gan-Li \cite{GL-coinduction} considered categories $\CCat=\FI_G$ for some group $G$, encoding representations of Wreath products $S_n\wr G$. Casto \cite{Ca} extended their treatment, and defined character polynomials in this context.
\item Barter \cite{Ba} considered the category $\CCat=\catname{T}$ of rooted trees with root-preserving embeddings.
\end{enumerate}
This approach has been further applied to topology, arithmetic and classical representation theory (see the respective citations).

Other generalizations considered categories of dimension zero, studied by Wiltshire-Gordon and Ellenberg (see \cite{WG} with applications in \cite{WGE}); homogeneous categories, studied by Randal-Williams and Wahl (see \cite{RWW}); and modules over twisted commutative algebras, studied by Sam-Snowden (see \cite{SS-tca}). We will not discuss these ideas here.

\bigskip In this paper we attempt to generalize and unify the treatments in Examples 1-4 and ask:
\begin{question}
What structure do these categories possess that supports the existence of a representation stability theory?
\end{question}
Here we offer an answer by fitting Examples 1-3 and others into the context of a broader theory: representation of categories \emph{of $\FI$ type}, i.e. categories that have structural properties similar to those of $\FI$ (see Definition \ref{def:FI-type} below). This approach is intended to subsume the individual treatments and eliminate the need to introduce a new theory in each specific case. At the same time, it allows one to consider new types of categories, such as the next example.

\begin{example}[\textbf{The categorical power $\FI^m$}]
As a first nontrivial example, and the original motivation behind this generalization, we consider the categorical powers $\FI^m$. These have objects that are (essentially) $m$-tuples $(n_1,\ldots,n_m)\in \N^m$ with automorphism groups the products $S_{n_1}\times\ldots\times S_{n_m}$. Such categories are the natural indexing category for various collections of linear subspace arrangements, to which our theory is applied in a companion paper \cite{Ga}. To see this at work consider the following example.

Fix $m\geq 1$ and let $\Rat^n(\Proj^1, \Proj^{m-1})$ be the space of based, degree $n$ rational maps ${\Proj^1\mor{}\Proj^{m-1}}$ that send $\infty$ to $[1:\ldots:1]$. This space admits an $(S_n\times\ldots\times S_n)$-covering $\PRat^n(\Proj^1, \Proj^{m-1})$ by picking an ordering on the zeros of the restrictions to the standard homogeneous coordinates functions on $\Proj^{m-1}$. The coverings fit naturally into a (contravariant) $\FI^m$-diagram of spaces, and their cohomology is an $\FI^m$-module. 

The groups $\Ho^i(\Rat^n(\Proj^1, \Proj^{m-1});\Q)$ can then be computed from the invariant part of the $(S_n\times\ldots\times S_n)$-representation $\Ho^i(\PRat^n(\Proj^1, \Proj^{m-1});\Q)$ by transfer. Representation stability for $\FI^m$-modules then gives the following.
\begin{thm}[\textbf{Homological stability for $\Rat^n$ \cite[Theorem 6.24]{Ga}}]\label{thm:intro-homological-stability}
For every $i\geq 0$ the $i$-th Betti number of $\Rat^n(\Proj^1, \Proj^{m-1})$ does not depend on $n$ for all $n\geq i$.
\end{thm}
\end{example}
In \S\ref{sec:rep_theory_of_FI} we discuss representation stability for $\FI^m$, which allows one to make such claims as Theorem \ref{thm:intro-homological-stability}. We remark that similar treatment could be applied to any product of categories whose representation stability is understood, but we do not pursue other examples here.

\subsection{Categories of $\FI$ type and free modules}
As outlined above, we are looking for categorical structure that gives rise to character polynomials satisfying Fact \ref{fact:stabilization}. We propose the following.
\begin{definition}[\textbf{Categories of $\FI$ type}] \label{def:FI-type}
We say that a category $\CCat$ is \emph{of $\FI$ type} if it satisfies the following axioms.
\begin{enumerate}
\item $\CCat$ is locally finite, i.e. all hom-sets are finite\footnote{Finiteness is not strictly necessary for many of the definitions and subsequent results. The author will be very interested to see how far one can push this theory with infinite automorphism groups}.
\item Every morphisms is a monomorphism, and every endomorphisms is an isomorphism.
\item For every pair of objects $c$ and $d$, the group of automorphisms $\Aut_\CCat(d)$ acts transitively on the set $\Hom_\CCat(c,d)$.
\item For every object $d$ there exist only finitely many isomorphism classes of objects $c$ for which $\Hom_\CCat(c,d)\neq \emptyset$ (we denote this by $c\leq d$).
\item $\CCat$ has pullbacks and pushouts\footnote{Pushouts are not quite what we want here, as these typically do not exists when one insists that all morphisms be injective. We replace this notion by \emph{weak pushouts}, defined below.}.
\end{enumerate}
\begin{remark}
Categories that satisfy the second half of condition 2 -- where every endomorphism is an isomorphism -- are called $\catname{EI}$ categories. The representation stability of such categories satisfying additional combinatorial conditions was studied by Gan-Li in \cite{GL-EI}.
\end{remark}

We will denote the automorphism group of an object $c$ by $G_c$.
\end{definition}

In \S\ref{subsec:char_poly} we define the collection of \emph{character polynomials} for a general category $\CCat$ of $\FI$ type - these are certain $\C$-valued class functions simultaneously defined on all automorphism groups $G_c$. Briefly, character polynomials are linear combinations of functions of the form $\binom{X}{\lambda}$ where $\lambda\subset G_c$ is some fixed conjugacy class. $\binom{X}{\lambda}$ evaluates on $g_d\in G_d$ to give the number of ways $g_d$ can be restricted to an element $g_c\in\lambda$, i.e. via morphisms $c\mor{f}d$ for which $g_d\circ f = f\circ g_c$ with $g_c\in \lambda$.
 
However, it is not at all clear that these functions satisfy the analog of Fact \ref{fact:stabilization}, or even that they can be reasonably thought of as polynomials, i.e. closed under taking products. To demonstrate these fundamental properties we propose a categorification of character polynomials, similar to the way in which group representations categorify class functions. Our categorification takes the form of \emph{free $\CCat$-modules}, introduced in Section \S\ref{sec:free-modules}.
\begin{definition}[\textbf{Free $\CCat$-modules}]\label{def:intro_free}
A $\CCat$-module is said to be \emph{free} if it is a direct sum of modules of the form $\Ind_c(V)$, where $\Ind_c$ is the left-adjoint functor to the restriction $M_\bullet \mapsto M_c$.

\begin{note}
Since we are only discussing finitely-generated $\CCat$-modules, free modules will always be taken to be \emph{finite} direct sums. Over the field of complex numbers these $\CCat$-modules are precisely the finitely-generated, projective ones.
\end{note}
\end{definition} 
This choice of categorification is justified by the 
following observation.
\begin{thmx}[\textbf{Categorification of character polynomials}]\label{thm:intro-categorification}
If $M_\bullet$ is a free $\CCat$-module over $\C$, then there exists a character polynomial $P$ whose restriction to $G_c$ coincides with the character of $M_c$ for every object $c$.

Conversely, the character polynomials that arise in this way span the space of all character polynomials on $\CCat$, defined in \S\ref{subsec:char_poly} below.
\end{thmx}

The structure of $\FI$ type then ensures that the class of free $\CCat$-modules, and subsequently character polynomials, has the properties that ultimately produce representation stability.
\begin{thmx}[\textbf{The class of free $\CCat$-modules}]\label{thm:intro-free-modules}
If $\CCat$ is a category of $\FI$ type, then the class of (finitely-generated) free $\CCat$-modules over $\C$ has the following properties:
\begin{enumerate}
\item The tensor product of two free $\CCat$-modules is again free.
\item There is a \emph{degree} filtration on free $\CCat$-modules, taking values in the objects of $\CCat$. Direct sums and tensor products act on this degree in the usual way with respect to an order relation $\leq$ on $\CCat$ and object addition $+$ defined below.
\item Every free $\CCat$-module $M_\bullet$ has a dual $\CCat$-module $M^*_\bullet: c\mapsto \Hom_{\C}(M_c,\C)$, which is again free of the same degree.
\item If $M_\bullet$ is a free $\CCat$-module of degree $\leq c$, then for every object $d\geq c$ the coinvariants $(M_d)/G_d$ are canonically isomorphic.
\end{enumerate}
\end{thmx}
This statement -- especially closure under tensor products -- is nontrivial and depends critically on the structure of $\FI$ type. For example, the specialization to $\FI$-modules was proved in \cite{CEF} using the projectivity of $\FI\#$-modules, and is related to the fact that products of binomial coefficients $\binom{n}{k}\binom{n}{l}$ can be expressed as linear combinations of $\binom{n}{r}$ with $r\leq k+l$.

\begin{remark}[\textbf{Working over different fields}]
Most of the results in Theorem \ref{thm:intro-free-modules} are set-theoretic in nature and follow from combinatorial properties of $\CCat$-sets. They thus hold in greater generality with the base field $\C$ replaced with an arbitrary commutative ring $R$. However, when trying to decategorify and conclude character-theoretic results, the assumption of characteristic $0$ becomes necessary. To simplify our exposition, we will phrase the results only for $\CCat$-modules over $\C$.
\end{remark}

Theorem \ref{thm:intro-free-modules} in particular gives the categorified analog of Fact \ref{fact:stabilization}. This fact captures the stabilization of the sequence of representations, as we shall see in the Application 1 below.
\begin{cor}[\textbf{$\boldsymbol{\Hom}$ stabilization}]\label{cor:inro-hom-stabilization}
If $M_\bullet$ and $N_\bullet$ are free $\CCat$-modules of respective degrees $\leq c_1$ and $\leq c_2$, then the spaces
\begin{equation}\label{eq:intro_tensor_hom}
\Hom_{G_d}(M_d,N_d) \cong (M^* \otimes N)_d/G_d
\end{equation}
are canonically isomorphic for all $d\geq c_1+c_2$.
\end{cor}
When the objects of $\CCat$ are parameterized by natural numbers, the addition $c_1+c_2$ is the usual addition operations. For the general definition of addition on objects, see Definition \ref{def:sums} below. Note that the identification of the two sides in Equation \ref{eq:intro_tensor_hom} is where characteristic $0$ assumption is used.

Decategorifying back to characters, one obtains the following.
\begin{cor}[\textbf{Inner product stabilization}]
If $P$ and $Q$ are character polynomials of respective degrees $\leq c_1$ and $\leq c_2$, then the inner products
\begin{equation}
\langle P, Q\rangle_{G_d} = \frac{1}{|G_d|}\sum_{g\in G_d} \bar{P}(g)Q(g)
\end{equation}
become independent of $d$ for all $d\geq c_1+c_2$.
\end{cor}
These claims will be proved in \S\ref{sec:free-modules} and \S\ref{sec:coinvariants}.

\subsection{Application 1: Stabilization of irreducible multiplicities}
Let $G$ be a finite group. Recall that over $\C$ the irreducible decomposition of a $G$-representation can be detected by $G$-intertwiners. Explicitly, if $V$ is a $G$-representation and $W$ is an irreducible representation, then the multiplicity at which $W$ appears in $V$ is $\dim \Hom_G (W,V)$. Similarly, if $V = \oplus_{i} W_i^{r_i}$ is an irreducible decomposition then
$$
\dim \Hom_G(V,V) = \sum_{i} r_i^2.
$$

Corollary \ref{cor:inro-hom-stabilization} then demonstrates that these dimensions stabilize in the case of free $\CCat$-modules.
\begin{cor}[\textbf{Stabilization of irreducible decomposition}]
Let $M_\bullet$ be a free $\CCat$-module of degree $\leq c$. At every object $d$ let
$$
M_d = \oplus_i W(d)_i^{r(d)_i}.
$$
be an irreducible decomposition. Then the sums $\sum_{i} r(d)_i^2$ do not depend on $d$ for $d\geq c+c$. More generally, if $N_\bullet$ is any other free $\CCat$-module of degree $\leq c'$ with irreducible decompositions
$$
N_d = \oplus_i W(d)_i^{s(d)_i}
$$
then the sums $\sum_i r(d)_i\cdot s(d)_i$ do not depend on $d$ for all $d\geq c+c'$.
\end{cor}
By choosing the test module $N_\bullet$ carefully, one can gain more information as to the individual multiplicities $r(d)_i$. In particular, it is often possible to relate the irreducible representations of the different groups $G_d$ and show that the individual multiplicities in fact stabilize for all $d\geq c+c$.

\subsection{Application 2: The category of $\CCat$-modules is Noetherian}
One of the most important themes in representation stability is the Noetherian property: the subcategory of finitely-generated modules is closed under taking submodules. This allows one to apply tools from homological algebra to finitely-generated modules, with far reaching applications (see e.g. \cite{CEF} and \cite{ICM}).

\begin{example}[\textbf{Configuration spaces of manifolds \cite{CEF}}]
Let $M$ be an orientable manifold. For every finite set $S$ the space of $S$-configurations on $M$, $\PConf^S(M)$, is the space of injections from $S$ to $M$. The functor $S\mapsto \PConf^S(M)$ is an $\FI^{\op}$-space, and $S\mapsto\Ho^i(\PConf^S(M);\Q)$ is an $\FI$-module.

Totaro \cite{To} proved that there is a spectral sequence converging to $\Ho^i(\PConf^S(M))$ and \cite{CEF} showed that the every $E^{p,q}_2$-term of this sequence is a finitely-generated $\FI$-module. \cite{CEF} also prove that $\FI$-modules over $\Q$ is Noetherian, and therefore finite-generation persists to the $E_\infty$-page, and subsequently to $\Ho^i$. Therefore the sequence $\Ho^i(\PConf^S(M);\Q)$ exhibits representation stability. One direct result is that the $i$-th Betti number of $\PConf^S(M)$ is eventually polynomial in $|S|$.
\end{example}

A corollary of the theory developed here is that the same Noetherian property holds in general.
\begin{thmx}[\textbf{The category $\Mod{\CCat}$ is Noetherian}]\label{thm:intro-noetherian}
If $\CCat$ is a category of $\FI$ type, then the category of $\CCat$-modules over $\C$ is Noetherian. That is, every submodule of a finitely generated $\CCat$-module is itself finitely generated.
\end{thmx}
Theorem \ref{thm:intro-noetherian}, proved below in \S\ref{sec:noetherian}, simultaneously generalizes the results by Church-Ellenberg-Farb \cite[Theorem 1.3]{CEF} and independently by Sam-Snowden \cite[Theorem 1.3.2]{SS-GL-mod}, who proved that the category of $\FI$-modules is Neotherian; Putman-Sam \cite{PS}, who proved the same for the category of $\VI$-modules; and Wilson \cite{Wi}, who proved this for $\FI_{\mathcal{W}}$-modules. 

Gan-Li \cite{GL-EI} generalized all of these Noetherian results and found that they hold for every category with a skeleton whose objects are parameterized by $\N$ satisfying certain combinatorial conditions see (\cite[Theorem 1.1]{GL-EI}). However, their theory does not address categories whose objects are not parameterized by $\N$, such as $\FI^m$ treated in \S\ref{sec:rep_theory_of_FI} below.

\bigskip One reason Noetherian results are important in our context is that they ensure that finitely-generated $\CCat$-modules exhibit the same stabilization phenomena as with free $\CCat$-module discussed in Application 1 (although without the effective bounds on stable range).
\begin{thmx}[\textbf{Stabilization of finitely-generated $\CCat$-modules}]
If $M_\bullet$ is a finitely-generated $\CCat$-module, then the sequence of coinvariants $M_c/G_c$ is eventually constant, i.e. all induced maps are $M_c/G_c\mor{}M_d/G_d$ are isomorphisms for sufficiently large objects $c$.

More generally, for every free $\CCat$-module $F_\bullet$ the sequence of spaces $\Hom_{G_c}(F_c,M_c)$ is eventually constant in the same sense.
\end{thmx}

\subsection{Application 3: Free modules in topology}
Beyond the applications of free $\CCat$-modules to the representation theory of the category $\CCat$, they also appear explicitly in topology. In a companion paper \cite{Ga} we consider the cohomology of $\CCat$-diagrams of linear subspace arrangements, for which we show that the induced cohomology $\CCat$-module is free. An immediate consequence, stated here somewhat informally, is the following (see \cite{Ga} for the precise definitions and statements).
\begin{cor}[\textbf{Stability of $\CCat$-diagrams of linear subspace arrangements}]\label{cor:intro-application-to-rational-maps}
Every (contravariant) $\CCat$-diagram $X^\bullet$ of linear subspace arrangements, that is generated by a finite collection of subspaces, exhibits cohomological representation stability. That is, for every $i\geq 0$ the $\CCat$-module $\Ho^i(X^\bullet;\C)$ is free. In particular, there exists a single character polynomial $P_i$ of $\CCat$ that uniformly describes the $G_c$-representation $\Ho^i(X^c;\C)$ for every object $c$.

Moreover, the respective quotients $X^c/G_c$ exhibit homological stability for $\C$-coefficients, and for various systems of constructible sheaves.
\end{cor}

\subsection{Acknowledgments} I wish to thank Benson Farb and Jesse Wolfson for many helpful conversations and suggestions that helped shape this work into its current form. I also thank Kevin Casto for explaining the arguments in Gan-Li's paper and how my work fits in with theirs.

\section{Preliminaries}\label{sec:prelim}
Let $\CCat$ be a category. Objects of $\CCat$ will typically be denoted by $c$, $d$, and so on. 

The categories with which we shall be working will have only injective morphisms. This typically precludes the possibility of having push-out objects. The following definition provides a means for salvaging some notion of a push-out diagram subject to this constraint.
\begin{definition}[\textbf{Weak push-out}]\label{def:weak-pushout}
A \emph{weak push-out} diagram is a pullback diagram
\squareDiagram{p}{c_1}{c_2}{d}{\tilde{f}_1}{\tilde{f}_2}{f_1}{f_2}
with the following universal property: for every other pullback diagram
\squareDiagram{p}{c_1}{c_2}{z}{\tilde{f}_1}{\tilde{f}_2}{h_1}{h_2}
there exists a unique morphism $d\mor{h} z$ that makes all the relevant diagrams commute.
We call $d$ the \emph{weak push-out object} and denote it by $c_1\coprod_p c_2$. The unique map $h$ induced from a pair of maps $c_i\mor{h_i}z$ is denoted by $h_1\coprod_p h_2$.
\end{definition}
This is similar to a usual push-out, but with ``all" commutative squares replaced by only pullback squares. When starting from a category that has push-outs, such as $\Set$ and $\Vect_k$, and passing to the subcategory that includes only injective maps, we lose the push-out structure. However, weak push-outs persist, and retain most of the same function.

\bigskip A standard notation that we will use throughout is the following.
\begin{definition}
We say that $c\leq d$ if there exists morphisms $c\mor{}d$.
\end{definition}
In categories of $\FI$ type (see Definition \ref{def:FI-type} above) this preorder relation between objects is essentially an order, i.e. if $c\leq d$ and $d\leq c$ then every morphism $c\mor{}d$ is invertible (for an explanation see \cite[Lamma 3.28]{Ga}). However, as noted in part (5) of Definition \ref{def:FI-type}, push-outs typically don't exist in categories of $\FI$-type and we adjust the definition by demanding the following property instead.

\begin{definition}[\textbf{Categories of $\FI$ type}]
A category $\CCat$ is said to be \emph{of $\FI$ type} is it satisfies axioms (1)-(4) from Definition \ref{def:FI-type}, and in addition:
\begin{enumerate}
\item[5.] $\CCat$ has pullbacks and weak push-outs, i.e. for every pair of morphisms $p\mor{f_i}c_i$ there exists a weak push-out $c_1\coprod_p c_2$; and for every pair $c_i \mor{g_i} d$ there exists a pullback $c_1 \times_{d} c_2$.
\end{enumerate}
\end{definition}
It seems possible that some of the theory should carry over to compact groups or even to semi-simple groups, but this direction will not be perused here.

\bigskip The primary objects of study are the representation of $\CCat$. These are the $\CCat$-modules defined in Definition \ref{def:intro_modules}. Our goal is to understand the category of $\CCat$-modules and relate it to the categories of representations of the individual automorphism groups $G_c$.

\subsection{Binomial sets and Character Polynomials} \label{subsec:char_poly}
The character polynomials for the symmetric groups are class functions simultaneously defined on $S_n$ for all $n$. These objects are closely linked with the phenomenon of representation stability in that the character of a representation-stable sequence is eventually given by a single character polynomial (see \cite{CEF}). We will now define character polynomials for a general category of $\FI$ type.

The following notion generalizes the collection of subset of size $k$ inside a set of $n$ elements. It will be used below in the definition of character polynomials.
\begin{definition}[\textbf{Binomial set}]
\label{def:binom_set}
Let $c$ and $d$ be two objects of $\CCat$. The group of automorphisms $G_c$ acts on $\Hom_\CCat(c,d)$ on the right by precomposition. Denote the quotient $\Hom_\CCat(c,d)/G_c$ by $\binom{d}{c}$. We will call this the \emph{binomial set, $d$ choose $c$}. If $c\mor{f}d$ is a morphism, we denote its class in $\binom{d}{c}$ by $[f]$.

Since the set $\Hom_\CCat(c,d)$ admits a left action by $G_d$, and this action commutes with the right action of $G_c$, the binomial set $\binom{d}{c}$ acquires a $G_d$ action naturally by $\sigma([f])=[\sigma\circ f]$. 
\end{definition}
Note that in the case of $\CCat=\FI$, the category of finite sets and injections, the binomial set $\binom{n}{k}$ is naturally in bijection with the collection of size $k$ subsets of $n$ (hence the terminology). Replacing $\FI$ by $\VI_\F$, the category of finite dimensional $\F$-vector spaces and injective linear functions, the binomial set $\binom{n}{k}$ is naturally the Grassmanian of $k$-planes in $\F^n$.

\begin{definition}[\textbf{Character polynomial}] \label{def:char_poly}
Let $c$ be an object of $\CCat$ and $\mu\subseteq G_c$ a conjugacy class. In this case we will denote $|\mu|=c$. The \emph{indicator character polynomial} of $\mu$ is the $\C$-valued class function $\binom{X}{\mu}$ simultaneously defined on all $G_d$ by
\begin{equation}\label{eq:indicator_poly}
\binom{X}{\mu}: \left(\sigma\in G_d\right) \mapsto  \left|\left\{ [f]\in\binom{d}{c} \mid \exists \psi\in \mu \text{ s.t. } \sigma\circ f = f\circ \psi \right\}\right|.
\end{equation}
The \emph{degree} of $\binom{X}{\mu}$ is defined to be $\deg(\binom{X}{\mu}):=|\mu|$.

A \emph{character polynomial} $P$ is a $\C$-linear combination of such simultaneous class functions. We say that the \emph{degree} of $P$ is $\leq d$ for an object $d$ if for every indicator $\binom{X}{\mu}$ that appears in $P$ nontrivially we have $|\mu|\leq d$. We denote this by $\deg(P)\leq d$.
\end{definition}

The following lemma shows that the above definition indeed gives rise to well-defined class functions.
\begin{lem}
The function $\binom{X}{\mu}$ is a class function of every group $G_d$. Furthermore, its definition in Equation \ref{eq:indicator_poly} does not depend on the choice of representative $f\in [f]$.
\end{lem}
\begin{proof}
First to see that Equation \ref{eq:indicator_poly} does not depend on the choice of $f$, suppose $f'$ is another representative of $[f]\in \binom{d}{c}$. Then there exists some $g\in G_c$ such that $f'=f\circ g$. Then for every $\sigma\in G_d$ and $\psi\in\mu$
$$
\sigma\circ f = f\circ \psi \iff \sigma\circ f' = (f\circ \psi) \circ g = f' \circ (g^{-1} \psi g)
$$
and $g^{-1} \psi g$ belongs to $\mu$ as well.

Lastly, to see that we get a class function take $\sigma' = h \sigma h^{-1}$. Then $[f]\in \binom{d}{c}$ satisfies $\sigma\circ f = f\circ \psi$ if and only in $[h\circ f]$ satisfies
$$
\sigma'\circ (h\circ f) = h \sigma h^{-1} (h\circ f) = h \circ (f\circ \psi) = (h\circ f)\circ \psi.
$$
If we denote by $U_\mu(\sigma)$ the set of classes $[f]\in \binom{d}{c}$ which is counted in Equation \ref{eq:indicator_poly} for $\sigma$, then we see that $U_\mu(h\sigma h^{-1})=h(U\mu(\sigma))$, and in particular these sets have equal cardinality.
\end{proof}
\begin{example}[\textbf{$\FI$ character polynomials}]\label{ex:FI_characters}
For $\CCat=\FI$ the automorphism group of the object $n=\{0,1,\ldots n-1\}$ is the symmetric group $S_n$. For any $k\in \N$ a conjugacy class in $S_k$ is described by a cycle type, $\mu = (\mu_1,\mu_2,\ldots,\mu_k)$ where $\mu_i$ is the number of $i$-cycles. For any other $n\in \N$, if we denote by $X_i$ the class function on $S_n$
$$
X_i(\sigma) = \# \text{ of $i$-cycles in }\sigma
$$
then we claim that
\begin{equation}\label{eq:FI_indicator_poly}
\binom{X}{\mu}(\sigma) = \binom{X_1(\sigma)}{\mu_1}\ldots \binom{X_k(\sigma)}{\mu_k}.
\end{equation}

Indeed, the class $[f]\in \binom{n}{k}$ of an injection $k\mor{f}n$ corresponds to the subset $Im(f)\subseteq n$. The condition that $\sigma\circ f = f\circ \psi$ translates to saying that $Im(f)$ is invariant under $\sigma$ and that the induced permutation on this subset has cycle type $\mu$. Then for a given $\sigma\in S_n$ the right-hand side of Equation \ref{eq:FI_indicator_poly} counts the number of ways to assemble such an invariant subset from the cycles of $\sigma$.
\end{example}
\cite{CEF} give Equation \ref{eq:FI_indicator_poly} as the definition of $\binom{X}{\mu}$. Thus our definition of character polynomials extends the classical notion of character polynomials for the symmetric groups to other classes of groups.
\begin{example}[\textbf{$\VI$ character polynomials}]
For $\CCat=\VI_\F$ the automorphism group of the object $[n] = \F^n$ is $\operatorname{GL}_n(\F)$. We describe the degree $1$ indicators. A conjugacy class in $\operatorname{GL}_1(\F)=\F^\times$ is just a non-zero element $\mu\in \F$. For every $n\in \N$ the function $\binom{X}{\mu}$ on a matrix $A\in\operatorname{GL}_n(\F)$ is given by
\begin{equation}
\binom{X}{\mu}(A) = \# \text{ of 1D eigenspaces of $A$ with eigenvalue } \mu.
\end{equation}
These are the $\VI$ analogs of $X_1$ on $S_n$, which counts the number of fixed points of a permutation.
\end{example}

\section{Free $\CCat$-modules}\label{sec:free-modules}
This section is devoted to defining free $\CCat$-modules and proving that, when $\CCat$ is of $\FI$ type, these modules satisfy the fundamental properties stated in Theorem \ref{thm:intro-free-modules}. Note that the statements and proofs in this section are essentially set-theoretic in nature, and therefore hold in a more general setting of $\CCat$-modules over any ring $R$. For concreteness we will only describe here the results with $R=\C$.

Free $\CCat$-modules are defined using a collection of left-adjoint functors. For every object $c$ there is a natural restriction functor
\begin{equation}
\Mod{\CCat} \mor{\Res_{c}} \Mod{\C[G_c]}\, , \; M_\bullet\mapsto M_c
\end{equation}
Following \cite{tD}, this functor admits a left-adjoint as follows.
\begin{definition} [\textbf{Induction $\CCat$-modules}] \label{def:Ind_functor}
Let $\Ind_c: \Mod{\C[G_c]} \mor{} \Mod{\CCat}$ be the functor that sends a $G_c$-representation $V$ to the $\CCat$-module
\begin{equation}
\Ind_c(V)_\bullet = \C[\Hom(d,\bullet)]\otimes_{G_d} V
\end{equation}
where morphisms in $\CCat$ act on these spaces naturally through their action on $\Hom(c,\bullet)$.

We call a $\CCat$-module of this form an \emph{induction module} of degree $c$, and denote $$\deg(\Ind_c(V))=c.$$
\end{definition}
\cite{tD} shows that the functor $\Ind_c$ is a left adjoint to $\Res_c$. Recall that in Definition \ref{def:intro_free} we called direct sum of induction modules free. The following additional terminology will also be useful.
\begin{definition} [\textbf{Degree of a free module}]
We say that a free $\CCat$-module $M_\bullet$ has \emph{degree} $\leq d$ if for every induction module $\Ind_c(V)$ that appears in $M_\bullet$ nontrivially we have $c\leq d$. In this case we denote $\deg(M_\bullet)\leq d$.

A \emph{virtual} free $\CCat$-module is a formal $\C$-linear combination of induction modules, e.g.
$$
\oplus_{i=1}^n \lambda_i \Ind_{c_i}(V_i) \; \text{where } \lambda_i\in \C.
$$
We extend the induction functors $\Ind_c$ linearly to virtual $G_c$-representations, i.e.
$$
\Ind_c( \oplus \lambda_i V_i ) := \oplus \lambda_i \Ind_c (V_i).
$$
\end{definition}
We propose that (virtual) free $\CCat$-modules are a categorification of character polynomials, much like the case for any finite group $G$ where (virtual) $G$-representations categorify class functions on $G$ .
\begin{definition}[\textbf{The character of a $\CCat$-module}]
If $M_\bullet$ is a $\CCat$-module, its \emph{character} is the simultaneous class function
$$
\chi_M : \coprod_{c} G_c \mor{} \C
$$
that for every object $c$ sends the group $G_c$ to the character of the $G_c$-representation $M_c$.
\end{definition}

One can express the character of induction modules in terms of indicator character polynomials, as follows.
\begin{lem}[\textbf{Character of induction modules}]\label{lem:character_of_ind}
If $V$ is any $G_c$-representation whose character is $\chi_V$, then the character of $\Ind_c(V)$ is given by
\begin{equation}\label{eq:character_of_ind}
\chi_{\Ind_c(V)} = \sum_{\mu \in \operatorname{conj}(G_c)} \chi_V(\mu) \binom{X}{\mu} 
\end{equation}
where $\operatorname{conj}(G_c)$ is the set of conjugacy classes of $G_c$, and $\chi_V(\mu)$ is the value $\chi_V$ takes on any $g\in\mu$. In particular we see that the character of $\Ind_c(V)$ is a character polynomial of degree $c$.
\end{lem}
\begin{proof}
Since all morphisms in $\CCat$ are monomorphisms, it follows for every object $d$ the equivalence class $f\circ G_c = [f]\in \binom{d}{c}$ is a right $G_c$-torsor. Thus there is an isomorphism of vector spaces
\begin{equation} \label{eq:ind_decomposition}
\Ind_c(V)_d = \C[\Hom_\CCat(c,d)]\topotimes[G_c] V = \topoplus[{[f]\in \binom{d}{c}}] \C([f])\topotimes[G_c]V \cong \topoplus[{[f]\in \binom{d}{c}}] V
\end{equation}
where the group $G_d$ permutes the summands through its action on $\binom{d}{c}$. It follows that the trace of $\sigma\in G_d$ gets a contribution from the summand $\C([f])\otimes_{G_c} V$ if and only if $\sigma([f])=[f]$. Consider such $[f]\in \operatorname{Fix}(\sigma)$, i.e. there exists some $\psi\in G_c$ such that $\sigma\circ f = f\circ \psi$. We get a commutative diagram
\squareDiagram{\C([f])\otimes_{G_c}V}{\C([f])\otimes_{G_c}V}{V}{V}{\sigma}{\cong}{\cong}{\psi}
so the trace of $\sigma|_{\C([f])\otimes_{G_c}V}$ is precisely $\chi_V(\psi)$. We get a formula for the character
\begin{equation}
\chi_{\Ind_c(V)}(\sigma) = \sum_{\substack{[f]\in \binom{d}{c}\\ \exists\psi (\sigma\circ f = f\circ \psi)}} \chi_V(\psi).
\end{equation}
Arranging this sum according to the conjugacy class of $\psi$ we get the equality claimed by Equation \ref{eq:character_of_ind}.
\end{proof}

A corollary or Lemma \ref{lem:character_of_ind} is that free $\CCat$-modules indeed categorify character polynomials.
\begin{thm}[\textbf{Categoricifaction of character polynomials}] \label{thm:categorification}
Character polynomials of degree $\leq d$ are precisely the characters of virtual free $\CCat$-modules of degree $\leq c$.
\end{thm}
\begin{proof}
It is sufficient to show that every $\binom{X}{\mu}$ is the character of some virtual free $\CCat$-module of degree $\leq |\mu|$. Denote $c=|\mu|$ and consider the indicator class function on $G_c$
$$
\chi_\mu(\psi)= \begin{cases}
1 \quad & \psi\in \mu \\
0 \quad & \psi\notin \mu.
\end{cases}
$$
Since the characters of $G_c$-representations form a basis for the class functions on $G_c$, there exist $G_c$-representations $V_1,\ldots,V_n$ and complex numbers $\lambda_1,\ldots,\lambda_n$ such that the virtual representation
$$
V_\mu = \oplus_{i=1}^n \lambda_i V_i
$$
has character $\chi_\mu$. Then by Lemma \ref{lem:character_of_ind} and linearity it follows that
$$
\chi_{\Ind_c(V_{\mu})} = \binom{X}{\mu}.
$$
\end{proof}

\subsection{Tensor products}
The categorification of pointwise products of character polynomials is the tensor product of free $\CCat$-modules. The goal of this subsection is to show that the product of two free modules is itself free.
\begin{definition}[\textbf{Tensor product of $\CCat$-modules}]
If $M_\bullet$ and $N_\bullet$ are two $\CCat$-module, their tensor product $\left(M\otimes N\right)_\bullet$ is the $\CCat$-module
\begin{equation}
\left(M\otimes N\right)_d = M_d \otimes N_d
\end{equation}
where a morphism $c\mor{f}d$ acts naturally by $M(f)\otimes N(f)$.
\end{definition}
At the level of characters, the tensor product corresponds to pointwise multiplication:
\begin{equation}
\chi_{M\otimes N} = \chi_M \cdot \chi_N.
\end{equation}

The main result of this subsection is the parts $(1)$ and $(2)$ of Theorem \ref{thm:intro-free-modules}. The following definition gives meaning to addition of objects so as to make the degree additive.
\begin{definition}[\textbf{Sum of objects}]\label{def:sums}
If $c_1$ and $c_2$ are two object of $\CCat$, then $c_1+c_2$ denotes the collection of objects $d$ that satisfy
$$
c_1\coprod_p c_2 \leq d
$$
for every weak push-out of $c_1$ and $c_2$. If $d$ belongs to the collection $c_1+c_2$ we denote $d\geq c_1+c_2$, i.e.
$$
d\geq c_1+c_2 \iff d\in c_1+c_2.
$$

If $M_\bullet$ is a free $\CCat$-module, we say that $\deg(M) \leq c_1+c_2$ if the degree is $\leq d$ for every $d\in c_1+c_2$.
\end{definition}
\begin{note}
If the collection $c_1+c_2$ contains an essential minimum object $d_0$, then we can identify $c_1+c_2$ with this minimum. In this case saying that $\deg(M)\leq c_1+c_2$ is equivalent to saying $\deg(M)\leq d_0$. In all the examples we currently know, the essential minimum object of $c_1+c_2$ is the weak coproduct $c_1\coprod_\emptyset c_2$. In particular, when $\CCat$ has a skeleton whose objects are parameterized naturally by $\N$ then the object ``$n_1+n_2$" coincides with the standard addition $n_1+n_2$ (hence the notation).
\end{note}

At the level of character polynomials Theorem \ref{thm:intro-free-modules}(1) translates into the following result.
\begin{cor}[\textbf{Closure under products}]
The collection of character polynomials forms an algebra under pointwise products, and the degree is additive with respect products. Namely, if $P$ and $Q$ are character polynomials of respective degrees $\leq c_1$ and $\leq c_2$, then their product $P\cdot Q$ is a character polynomial of degree $\leq c_1+c_2$.
\end{cor}
\begin{note}
It is not immediately clear that the product of two expressions $\binom{X}{\mu}$ and $\binom{X}{\nu}$ can be expanded in terms of other such expressions, but we now see they can. To demonstrate the nontriviality of this statement consider the standard binomial coefficients: for $X = \binom{X}{1}$ we have an expansion
$$
X\cdot \binom{X}{k} = (k+1)\binom{X}{k+1}+ k\binom{X}{k}
$$
\begin{problem}
Find general formula for the expansion of $\binom{X}{k_1}\binom{X}{k_2}$ in terms of $\binom{X}{k}$'s.
\end{problem}
\end{note}

The proof of Theorem \ref{thm:intro-free-modules}(1) will use the following definitions and lemmas. First we need an easy technical observation.
\begin{lem}[\textbf{Pull-back invariance}]\label{lem:pullback-invariance}
Suppose
\squareDiagram{a}{b_1}{b_2}{c}{f_1}{f_2}{g_1}{g_2}
is some commutative diagram in $\CCat$ and $c\mor{h}d$ is some monomorphism. By composing with $h$ we get another diagram
\squareDiagram{a}{b_1}{b_2}{d}{f_1}{f_2}{h\circ g_1}{h\circ g_2}

If one of these diagrams is a pull-back, then so is the other.
\end{lem}
Second we define the \emph{push-out set} of three objects.
\begin{definition}[\textbf{Push-out set}]
Let $c_1$ and $c_2$ be two objects of $\CCat$. For any object $d$ we define the \emph{push-out set} $\PO{d}{c_1,c_2}$ to be the set of pairs of morphisms $(c_i\mor{g_i}d\mid i=1,2)$ that present $d$ as a weak push-out of $c_1$ and $c_2$. That is to say that the pullback diagram
\squareDiagram{c_1\times_d c_2}{c_1}{c_2}{d}{}{}{g_1}{g_2}
is a weak push-out diagram.
\end{definition}
\begin{remark}
It is straightforward to verify that the procedure of replacing $d$ by an isomorphic object $d'$ and mapping $c_1$ and $c_2$ into $d'$ through any isomorphism ${d\mor{\sim}d'}$ preserves weak push-out diagrams. Therefore any such isomorphism induces a natural bijection of sets
$$
\PO{d}{c_1,c_2} \mor{\sim} \PO{d'}{c_1,c_2} 
$$
by left-composition. In particular, the group of automorphisms $G_d$ acts on $\PO{d}{c_1,c_2}$ on the left. Similarly, the group $G_{c_1}\times G_{c_2}$ acts naturally on the right by precomposition.
\end{remark}

The general philosophy of this work is the following: statements about representation stability (of which Theorem \ref{thm:intro-free-modules}(1) is one) are reflected by statement about $\CCat$-sets. Therefore closure under tensor products should be a consequence of a set-theoretic observation. This is the content of the next lemma.
\begin{lem}[\textbf{Tensor products: set version}] \label{lem:products_set_version}
Let $\CCat$ be a category of $\FI$ type. There is a natural isomorphism between the product functor
$$
\Hom(c_1,\bullet)\times \Hom(c_2,\bullet)
$$
and the disjoint union functor
$$
\coprod_{[d]} \Hom\left( d, \bullet \right)\times_{G_d} \PO{d}{c_1,c_2}
$$
where $[d]$ ranges over the isomorphism classes of $\CCat$ and $d$ is some representative of $[d]$. Furthermore this natural isomorphism respects the right $(G_{c_1}\times G_{c_2})$-action on the two functors.
\end{lem}
\begin{proof}
For any object $x$ and a representative $d$ of the isomorphism class $[d]$ we define a function
\begin{equation}
\Hom\left( d, x \right)\times_{G_d} \PO{d}{c_1,c_2} \mor{\Psi_x^d}  \Hom(c_1,x)\times \Hom(c_2,x)
\end{equation}
by composition, i.e.
\begin{equation}
\left[d\mor{f}x , (c_i\mor{r_i}d) \right] \mapsto  \left( c_i \mor{f\circ r_i} x \right)
\end{equation}
By the associativity of composition, this is well-defined on the product over $G_d$. Moreover, $\Psi^d_\bullet$ is clearly natural in $x$ and respects the right action of $G_{c_1}\times G_{c_2}$ given by precomposition.

Letting $d$ range over all isomorphism classes  we get a natural transformation from the union
\begin{equation}
\coprod_{[d]}\Hom\left( d, \bullet \right)\times_{G_d} \PO{d}{c_1,c_2} \mor{\Psi_\bullet}  \Hom(c_1,\bullet)\times \Hom(c_2,\bullet)
\end{equation}
which respects the right $G_{c_1}\times G_{c_2}$-action.

In the other direction, let $x$ again be any object. We define a function
\begin{equation}
\Hom(c_1,x)\times \Hom(c_2,x) \mor{\Phi^d_x} \Hom\left( d, x \right)\times_{G_d} \PO{d}{c_1,c_2} 
\end{equation}
as follows. Let $(c_i\mor{f_i}x\mid i=1,2)$ be a pair of morphisms. Construct their pull-back
\squareDiagram{p}{c_1}{c_2}{x}{\alpha_1}{\alpha_2}{f_1}{f_2}
and form the weak push-out for $p\mor{\alpha_i}c_i$
\squareDiagram{p}{c_1}{c_2}{c_1\coprod_p c_2=:d}{\alpha_1}{\alpha_2}{r_1}{r_2}
The universal property of the weak push-out then implies that there exists a unique morphism $d\mor{f}x$ such that $f\circ r_i = f_i$. We define $\Phi^d_x$ by
\begin{equation}
(c_i \mor{f_i}x )\mapsto \left[(d\mor{f}x), c_i\mor{r_i}d \right].
\end{equation}

To see that $\Phi^d_x$ is well-defined, suppose
\squareDiagram{p'}{c_1}{c_2}{d'}{\alpha_1'}{\alpha_2'}{r_1'}{r_2'}
is another weak push-out diagram produced by the same procedure and $d'\mor{f'}x$ is the corresponding induced map. First we observe that since $p$ and $p'$ are both pull-backs of the pair $(f_1,f_2)$ there exists an isomorphism $p\mor{\tau}p'$ for which $\alpha_i'\circ \tau =  \alpha_i$ for $i=1,2$. Second, we replace $p'$ by $p$ in the weak push-out diagram, mapping it though $\tau$, i.e.
\squareDiagram{p}{c_1}{c_2}{d'}{\alpha_1'\circ \tau}{\alpha_2'\circ \tau}{r_1'}{r_2'}
and this is again a weak push-out diagram. Therefore, by the universal property of the weak push-out, there exists a unique morphism $d\mor{\psi}d'$ for which $\psi\circ r_i = r_i'$. The same reasoning applied in reverse shows that $\psi$ admits a unique inverse, and therefore $d\cong d'$. Since we picked $d$ to be the representative for the isomorphism class $[d]$, it follows that $d=d'$ and that $\psi\in G_d$. The induced map $f$ is characterized by the property that $f\circ r_i = f_i$, and similarly for $f'$ and $r_i'$. Therefore we find that
$$
f_i = f'\circ r_i' = f'\circ \psi \circ r_i
$$
which by the universal property of $d$ shows that in fact $f = f'\circ \psi$. Our function $\Phi^d_x$ is defined as to send the pair $(f_1,f_2)$ to
$$
\left[ f, (c_i\mor{g_i}d) \right] = \left[ f'\circ \psi, (c_i\mor{g_i}d) \right] = \left[ f', (c_i\mor{\psi\circ g_i}d) \right] = \left[ f', (c_i\mor{g_i'}d) \right]
$$
which we now see that is uniquely defined.

The two functions $\Psi^d_x$ and $\Phi^d_x$ are clearly inverse, and therefor they together form a natural isomorphism between the two functors. As stated above, this isomorphism respects the right $G_{c_1}\times G_{c_2}$-action.
\end{proof}

Now we can prove that free $\CCat$-modules are indeed closed under tensor products.
\begin{proof} [Proof of Theorem \ref{thm:intro-free-modules}(1)]
By the distributivity of tensor products, it is enough to verify the claim for induction modules of respective degrees $\leq c_1$ and $c_2$ respectively. Moreover, by the transitivity of the order relation between objects, it will suffice if we assume that the degrees are precisely $c_1$ and $c_2$ respectively. Let $\Ind_{c_1}(V)$ and $\Ind_{c_2}(W)$ be two such $\CCat$-modules.

We apply an easy-to-verify equality of tensor products,
\begin{eqnarray*}
\Ind_{c_1}(V)_\bullet \otimes_k \Ind_{c_2}(W)_\bullet &=& \left(\C[\Hom{c_1,\bullet}]\otimes_{G_{c_1}} V\right) \otimes_k \left(\C[\Hom{c_2,\bullet}]\otimes_{G_{c_2}} W\right) \\
 &\cong& \left(\C[\Hom(c_1,\bullet)]\otimes_k \C[\Hom(c_2,\bullet)]\right) \bigotimes_{G_{c_1}\times G_{c_2}} \left( V \boxtimes W \right)
\end{eqnarray*}
and to this we can apply the natural isomorphism
\begin{equation}
\C[\Hom(c_1,\bullet)]\otimes \C[\Hom(c_2,\bullet)] \cong \C\left[ \Hom(c_1,\bullet)\times \Hom(c_2,\bullet) \right].
\end{equation}

In Lemma \ref{lem:products_set_version} we found a natural isomorphism between the product
$$
\Hom(c_1,\bullet)\times \Hom(c_2,\bullet)
$$
and the union
$$
\coprod_{[d]} \Hom\left( d, \bullet \right)\times_{G_d} \PO{d}{c_1,c_2}
$$
which when composed with the permutation representation functor $X \mapsto \C[X]$ yields a natural isomorphism
\begin{equation}
\C\left[ \Hom(c_1,\bullet)\times \Hom(c_2,\bullet) \right] \cong \bigoplus_{[d]} \C[\Hom(d,\bullet)]\otimes_{G_d} \C[\PO{d}{c_1,c_2}]
\end{equation}

By the associativity of the tensor product, we get a natural isomorphism
\begin{eqnarray*} \label{eq:tensor_of_free}
\Ind_{c_1}(V)_\bullet \otimes_k \Ind_{c_2}(W)_\bullet &\cong& \bigoplus_{[d]} \C[\Hom(d,\bullet)] \topotimes[G_d] \C[\PO{d}{c_1,c_2}] \bigotimes_{G_{c_1}\times G_{c_2}} (V\boxtimes W) \\
 &=& \bigoplus_{[d]} \Ind_d\left( \C[\PO{d}{c_1,c_2}] \bigotimes_{G_{c_1}\times G_{c_2}} (V\boxtimes W) \right) {}_\bullet
\end{eqnarray*}
as claimed.

Note that for $d$ to have a non-zero contribution to this direct sum, the set $\PO{d}{c_1,c_2}$ must be non-empty. In particular, there exists a decomposition $\displaystyle{d=c_1\coprod_p c_2}$. This proves the claim regarding the degree of terms in the sum. Since there are only finitely many isomorphism classes of objects with such a presentation, the above direct sum decomposition is finite.
\end{proof}

\subsection{Dualization}
One would like to define the dual of a $\CCat$-module $M_\bullet$ by $(M^*)_c = (M_c)^*$. Unfortunately, this will not be a $\CCat$-module in general (it will be a $\CCat^{\op}$-module). In this subsection we show that when dealing with free $\CCat$-modules there is a good notion of dualization.

\begin{definition}[\textbf{Dual $\CCat$-module}]
For an induction module $\Ind_c(V)$ we define its dual $\CCat$-modules by
\begin{equation}
\Ind_c(V)^* = \Ind_c(V^*)
\end{equation}
where $V^*$ is the $G_c$-representation dual to $V$. Extend this definition linearly to all (virtual) free $\CCat$-modules. 
\end{definition}

We claim that this indeed gives a good notion of duals.
\begin{thm}
If $M_\bullet$ is a free $\CCat$-module then there is a homomorphism of $\CCat$-modules
\begin{equation}
M_\bullet^* \otimes M_\bullet \mor{\ev} \C_\bullet
\end{equation}
where $\C_\bullet$ is the trivial $\CCat$-module with $\C_d = \C$ for every object. This pairing is non-degenerate and thus defines an isomorphism of $G_d$-representations $(M^*)_d\cong (M_d)^*$ for every object $d$.

We conclude that the dual of a free $\CCat$-module of degree $\leq c$ is again a $\CCat$-module of degree $\leq c$.
\end{thm}
\begin{proof}
Suppose $M = \oplus_i \Ind_{c_i}(V_i)$. Then there is a decomposition of $\CCat$-modules
$$
M^*\otimes M = \topoplus[i,j] \Ind_{c_i}(V_i^*)\otimes \Ind_{c_j}(V_j).
$$
We define the pairing to be $0$ for all $i\neq j$. For $i=j$ consider a single induction module $\Ind_c(V)$ and decompose it using Equation \ref{eq:ind_decomposition}
$$
\Ind_c(V)_d = \oplus_{[f]\in\binom{d}{c}} V \implies \left(\Ind_c(V^*)\otimes \Ind_c(V^*)\right)_d = \oplus_{[f],[g]\in\binom{d}{c}} V^*\otimes V.
$$
Set the pairing to be $0$ on all $[f]\neq [g]$, and for $[f]=[g]$ use the natural contraction on $V^*\otimes V$. This produces a map
$$
\oplus_{[f],[g]\in\binom{d}{c}} V^*\otimes V \mor{} \oplus_{[f]\in \binom{d}{c}} \C \mor{+} \C
$$
which is the pairing we sought.

Explicitly, the pairing on $\Ind_c(V^*)\otimes \Ind_c(V)$ is given by
\begin{equation}\label{eq:contraction}
\langle f\otimes v^*, g\otimes v \rangle = \sum_{\substack{\psi\in G_c \\ f\circ \psi = g}} \langle v^*, \psi(v)\rangle = \begin{cases}
\langle v^*, \psi(v)\rangle & f\circ \psi = g \\
0 & [f]\neq [g]
\end{cases}
\end{equation}
It is straightforward to check that the above pairing is invariant under the action of morphisms in $\CCat$. It is thus a morphism of $\CCat$-modules, as claimed. One can also check that the pairing is non-degenerate, and thus defines the claimed $G_d$-equivariant isomorphism
$$
(M^*)_d \mor{\sim} (M_d)^*.
$$
\end{proof}
\begin{cor}[\textbf{The $\Hom$ $\CCat$-module}] \label{cor:hom_module}
If $M_\bullet$ is a free $\CCat$-modules and $N_\bullet$ is any $\CCat$-module, then there exists a $\CCat$-module $\Hom(M,N)_\bullet$ whose value at $d$ is the $G_d$-representation $\Hom_\C(M_d,N_d)$.

A morphism $d\mor{f}e$ induces a function $\Hom_\C(M_d,N_d)\mor{f_*}\Hom_\C(M_e,N_e)$ satisfying the following naturality property: if $M_d\mor{T}N_d$ is any linear function, then there is a commutative diagram
\squareDiagram{M_d}{N_d}{M_e}{N_e}{T}{M(f)}{N(f)}{f_*(T)}

Furthermore, if $N_\bullet$ is itself free, and the degrees of $M_\bullet$ and $N_\bullet$
are $\leq c_1$ and $\leq c_2$ respectively, then $\Hom(M,N)_\bullet$ is also free and has degree $\leq c_1+c_2$.
\end{cor}
\begin{proof}
The desired $\CCat$-module is the tensor product $M^*\otimes N$. All other claims follow for the properties of tensor products and duals proved above.
\end{proof}

\section{The Coinvariant quotient and Stabilization}\label{sec:coinvariants}
When $G$ is a finite group, the coinvariants of a $G$-representation are the categorified analog of averaging over class function: if $\chi$ is the character of a $G$-representation $V$, then
\begin{equation}\label{eq:coinvariants_average}
\dim V_G = \frac{1}{|G|}\sum_{g\in G}\chi(g).
\end{equation}
Such averages appear in $G$-inner products, which we want to relate for the various automorphism groups $G_c$ of our category $\CCat$. This section will therefore analyze the behavior of free $\CCat$-modules under taking their coinvariants. Recall that the coinvariant quotient of a $G$-representation $V$ is its maximal invariant quotient, namely
$$
V_G = V/\langle v-gv \mid v\in V, g\in G \rangle.
$$
We will also denote this quotient by $V/{G}$.

In the context of a $\CCat$-module $M_\bullet$ we can form the $G_c$-coinvariant quotient of $M_c$ for every object $c$. If $c\mor{f}d$ is any morphism and $M_c\mor{M(f)}M_d$ the induced map, then it descends to a well-defined map on the coinvariants. Indeed, this follows from the assumptions that $G_d$ acts transitively on $\Hom_\CCat(c,d)$: if $g\in G_c$ is any automorphism, then $f$ and $f\circ g$ are two morphisms from $c$ to $d$ and thus there exists some $\tilde{g}\in G_d$ for which $\tilde{g}\circ f = f\circ g$. This shows that for every $v\in M_c$
$$
v-g(v) \overset{M(f)}{\longmapsto} f(v) - f\circ g(v) = f(v) - \tilde{g}\left(f(v)\right)
$$
and indeed $v-gv$ gets mapped to zero in the coinvariant quotient of $M_d$.
\begin{definition}[\textbf{The coinvariant quotient}]
We call the resulting $\CCat$-module of coinvariant quotients \emph{the coinvariant $\CCat$-module} of $M_\bullet$ and denote it by $(M/G)_\bullet$.
\end{definition}
\begin{note}
Every two morphisms $c\mor{f}d$ and $c\mor{f'}d$ give rise to the same map between coinvariants. This is because there exists some $\tilde{g}\in G_d$ for which $\tilde{g}\circ f = f'$ and this $\tilde{g}$ acts trivially on $(M/G)_d$. Thus for every pair $c\leq d$ there is a well-defined map between the coinvariants $(M/G)_c \mor{} (M/G)_d$.
\end{note}

The coinvariant quotient forms an endofunctor on $\CCat$-modules. In this subsection we study the action of this functor on free $\CCat$-modules and demonstrate that they exhibit stability under its operation.
\begin{lem} \label{lem:coinvariants_of_inductions}
Let $V$ be any $G_c$-representation and $\Ind_c(V)$ the corresponding induction module. The $G_d$-coinvariants of $\Ind_c(V)_d$ are given by
\begin{equation}
(\Ind_c(V)/G)_d \cong \begin{cases}
V/{G_c} &  \text{if } c\leq d \\
0 &  \text{otherwise}
\end{cases}
\end{equation}
with all morphisms $c\leq d\mor{f}d'$ inducing the identity map.
\end{lem}
\begin{remark}
This again reflects a statement about $\CCat$-sets. Namely, that the set of orbits $G_d\backslash\Hom_\CCat(c,d)$ is either a singleton if $c\leq d$ or empty otherwise.
\end{remark}
\begin{proof}[Proof of Lemma \ref{lem:coinvariants_of_inductions}]
Recall that the coinvariant quotient of a $G$-representation $W$ can be defined as the tensor product
\begin{equation}
(W)_G \cong \C\otimes_{G} W
\end{equation}
where $\C$ denotes the trivial $G$-representation.

Using the associativity of tensor products, and the presentation of $\Ind_c(V)$ as one, we get
\begin{equation}
\left(\Ind_c(V)/G\right)_d \cong \C \otimes_{G_d} \C[\Hom_\CCat(c,d)] \otimes_{G_c} V \cong \C[G_d\backslash\Hom_\CCat(c,d)] \otimes_{G_c} V
\end{equation}

By hypothesis the $G_d$ action on $\Hom_\CCat(c,d)$ is transitive. Therefore if $c\leq d$ then $\Hom(c,d)\neq \emptyset$ and this set forms a single orbit. Furthermore, a morphism $d\mor{f}d'$ carries this single orbit corresponding to $d$ to the one corresponding to $d'$. In the case where there are no morphisms $c\mor{}d$ we have the empty set. In other words we have
\begin{equation}
\C[G_d\backslash\Hom_\CCat(c,d)] \cong \begin{cases}
\C & \text{if } c\leq d \\
0 & \text{otherwise}
\end{cases}
\end{equation}
and a morphism $c\leq d\mor{f}d'$ induces the identity map on $\C$. Tensoring with $V$ over $G_c$ we get $V/{G_c}$ when $c\leq d$, zero otherwise, and morphisms as stated.
\end{proof}

Applying this result to direct sums of induction $\CCat$-modules, we can formulate what happens to free $\CCat$-modules when we take their coinvariants.
\begin{thm}[\textbf{Coinvariant stabilization}] \label{thm:coinvariant_stabilization}
When the coinvariants functor is applied to any free module of degree $\leq c$, all maps induced by $\CCat$-morphisms are injections, and all maps induced by morphisms between objects $\geq c$ are isomorphisms. 

Explicitly, the stable isomorphism type of the coinvariant quotient of a free $\CCat$-module $\oplus_{i} \Ind_{c_i}(V_i)$ is given by
\begin{equation}\label{eq:coinvariants_of_free}
\lim_{\bullet\rightarrow \infty}\left(\oplus_i \Ind_{c_i}(V_i)/G\right)_\bullet = \oplus_{i} V_i/G_{c_i}.
\end{equation}
\end{thm}

This translates to the following result regarding character polynomials.
\begin{cor}[\textbf{Stabilization of Expectation}]\label{cor:expectation_stabilization}
If $P$ is a character polynomial of degree $\leq c$, then its $G_d$-expected number
$$
\mathbb{E}_{G_d}[P] := \frac{1}{|G_d|}\sum_{\sigma\in G_d} P(\sigma)
$$
does not depend on $d$ for $d\geq c$.

Furthermore, if $P$ is the characters of free $\CCat$-modules then $\mathbb{E}_{G_d}[P]$ is a non-negative integer, monotonically increasing in $d$.
\end{cor}
\begin{proof}
Recall that for a $G$-representation $V$ the expectation
$$
\frac{1}{|G|}\sum_{g\in G} \Tr(g) = \Tr\left(\frac{1}{|G|}\sum_{g\in G} g\right)
$$
is the trace of the projection $V\twoheadrightarrow V^G$, whose existence also demonstrates that $V^G = V/G$. The expectation is thus $\dim_\C(V/G)$. In particular it is a non-negative integer.

Suppose $P$ is the character of the free $\CCat$-module $M_\bullet$ of degree $\leq c$. By Theorem \ref{thm:coinvariant_stabilization} the coinvariants $(M/G)_\bullet$ is a $\CCat$-module, all of whose induced maps are injections, and isomorphisms for objects $\geq c$. Thus the sequence of dimensions $\dim_\C (M/G)_d$ is monotonic in $d$ and becomes constant when $d\geq c$.

The general statement follows by linearity.
\end{proof}

We are often interested in the $G$-inner product of characters:
$$
\langle \chi_1, \chi_2 \rangle_G = \frac{1}{|G|}\sum_{g\in G} \chi_1(g) \bar{\chi}_2(g) = \mathbb{E}_G[ \chi_1\cdot\bar{\chi}_2 ]
$$
which is central to character theory. For character polynomials the previous corollary gives the following immediate stability statement.
\begin{cor}[\textbf{Stabilization of inner products}]
If $P$ and $Q$ are character polynomials of respective degrees $\leq c_1$ and $\leq c_2$, then the $G_d$-inner products
\begin{equation}
\langle P, Q\rangle_{G_d} = \frac{1}{|G_d|}\sum_{\sigma\in G_d} P(\sigma)\bar{Q}(\sigma)
\end{equation}
does not depend on $d$ for all $d\geq c_1+c_2$.

Furthermore, if $P$ and $Q$ are the characters of free $\CCat$-modules then $\langle P,Q\rangle_{G_d}$ is a non-negative integer, monotonic in $d$.
\end{cor}
\begin{proof}
The claim follows directly from the presentation
$$
\langle P, Q\rangle_{G_d} = \mathbb{E}_{G_d}[P\bar{Q}]
$$
and Corollary \ref{cor:expectation_stabilization}.

If $P$ and $Q$ are the characters of $M_\bullet$ and $N_\bullet$ then $P\bar{Q}$ is the character of the free $\CCat$-module $M\otimes N^*$. Integrality and monotonicity follow.
\end{proof}

\section{Noetherian property} \label{sec:noetherian}
In this section we apply the theory developed in the previous sections to prove that the category of $\CCat$-modules is Noetherian. Our proof strategy follows the argument made by Gan-Li in \cite{GL-EI}. The main theorem of this section is the following.
\begin{thm}[\textbf{$\Mod{\CCat}$ is a Noetherian category}]\label{thm:Noetherian}
Every $\CCat$-submodule of a finitely generated $\CCat$-module is itself finitely generated.
\end{thm}
Theorem \ref{thm:Noetherian} will be proved at the end of this section. First we need some preliminary results. We start with an extension result for equivariant homomorphisms between free $\CCat$-modules.
\begin{lem}[\textbf{Equivariant extension}]\label{lem:extensions}
Let $M_\bullet$ be a free $\CCat$-module. There is a left-exact endofunctor on $\CCat$-modules
\begin{equation}
N_\bullet \mapsto \Hom_{G_\bullet}(M_\bullet,N_\bullet)
\end{equation}
whose image is contained in trivial $\CCat$-modules. The value  of the module $\Hom_{G_\bullet}(M_\bullet,N_\bullet)$ at an object $d$ is the vector space $\Hom_{G_d}(M_d,N_d)$. In particular, for every $d\leq e$ there is a canonical map
\begin{equation}
\Hom_{G_d}(M_d,N_d)\mor{\Psi_d^e} \Hom_{G_e}(M_e,N_e)
\end{equation}
that promotes a $G_d$-linear map to a $G_e$-linear one.

Furthermore, if $N_\bullet$ is itself free, and the degrees of $M_\bullet$ and $N_\bullet$ are $\leq c_1$ and $\leq c_2$ respectively, then the extension map $\Psi_d^e$ is an isomorphism whenever $d\geq c_1+c_2$. In particular, equivariant morphisms extend uniquely in this range.
\end{lem}
\begin{proof}
To get the proposed endofunctor we use dualization, tensor products and coinvariants:
\begin{equation}
N_\bullet \mapsto \left(M^*\otimes N\right)_\bullet \mapsto \left[\left(M^*\otimes N\right)/G\right]_\bullet
\end{equation}
This gives rise to an endofunctor whose value at $d$ is
$$
\left(M^*_d\otimes N_d\right)/G_d.
$$
The tensor product is naturally isomorphic to $\Hom_\C(M_d,N_d)$ and averaging over $G_d$ gives a natural lift from coinvariants to invariants. Thus the value at $d$ is naturally isomorphic to
$$
\Hom_{\C}(M_d,N_d)^{G_d} = \Hom_{G_d}(M_d,N_d)
$$
and indeed the desired functor exists. Left exactness follows from the general fact that the functor $\Hom(M,\bullet)$ is left exact.

Lastly, if $M_\bullet$ and $N_\bullet$ are free of respective degrees $\leq c_1$ and $\leq c_2$ then by Theorem \ref{thm:intro-free-modules}(1) $M^*\otimes N$ is free of degree $\leq c_1+c_2$. We then apply Theorem \ref{thm:coinvariant_stabilization} and see that its coinvariants stabilize for all $d\geq c_1+c_2$ in the sense that all induced maps $\Psi_d^e$ are isomorphisms.
\end{proof}

When the range $N_\bullet$ is not free we cannot guarantee that the extension maps $\Psi_d^e$ be eventually isomorphisms. But in the case where the range is contained in a free module, we can at least salvage injectivity.
\begin{cor}[\textbf{Injective extension}]\label{cor:injective_extension}
If $M_\bullet$ and $N_\bullet$ are free $\CCat$-modules of respective degrees $\leq c_1$ and $\leq c_2$, and $X_\bullet\subseteq N_\bullet$ is any $\CCat$-submodule, then the extension maps
\begin{equation}
\Hom_{G_d}(M_d,X_d) \mor{\Psi_d^e} \Hom_{G_e}(M_e,X_e)
\end{equation}
are injective for all $e\geq d\geq c_1+c_2$.
\end{cor}
\begin{proof}
For every $e\geq d$ we have a commutative square of extensions
\squareDiagram{\Hom_{G_d}(M_d,X_d)}{\Hom_{G_d}(M_d,N_d)}{\Hom_{G_e}(M_e,X_e)}{\Hom_{G_e}(M_e,N_e)}{X_d\hookrightarrow N_d}{\Psi_d^e}{\Psi_d^e}{X_e\hookrightarrow N_e}
and since $M$ and $N$ are free of the given degrees, it follows that the rightmost extension map is an isomorphism when $d\geq c_1+c_2$. Furthermore, the two horizontal maps are injective by left-exactness. Thus we have a square in which all but the leftmost map are injections. This implies that the leftmost map is injective as well.
\end{proof}

We are now ready to prove that $\Mod{\CCat}$ has the Noetherian property.
\begin{proof}[Proof of Theorem \ref{thm:Noetherian}]
Suppose that $M_\bullet$ is a finitely generated $\CCat$-module and $X^0_\bullet\subseteq X^1_\bullet \subseteq \ldots \subseteq M_\bullet$ is an ascending chain of submodules. We need to show that $X^N = X^{N+1}=\ldots$ for some $N\in \N$. As in the standard proofs of Hilbert's Basis Theorem, we divide the task into two parts: controlling growth in all large degrees, then handling lower degrees using Noetherian property of finite direct sums.

We assume without loss of generality that $M_\bullet$ is a free, finitely-generated $\CCat$-module of degree $c$, as every finitely-generated $\CCat$-module is a quotient of a finite sum of such. For brevity we denote the functor $X_\bullet \mapsto \Hom_{G_\bullet}(M_\bullet, X_\bullet)$ by $F$, i.e.
$$
F(X)_d := \Hom_{G_d}(M_d, X_d).
$$
Since $M_\bullet$ is free of degree $\leq c$, it follows that all induced extension maps
$$
F(M)_d\mor{\Psi_d^e} F(M)_e
$$
are isomorphisms when $d\geq c+c$. Fix an object $d\geq c+c$. We get a collection of subspaces inside $F(M)_d$ by considering the images
\begin{equation}
\left\{F(X^n)_e \hookrightarrow F(M)_e \mor{(\Psi_d^e)^{-1}} F(M)_{d} \right\}_{n\in \N,\; e\geq d}
\end{equation}
Since $F(M)_{d}$ is itself Noetherian (a finite dimension vector space), this collection of subspaces has a maximal element, say the image of $F(X^{N_0})_{e_0}$.
\begin{claim}
For all $n\geq N_0$ and $e\geq e_0$ we have $X^n_e = X^{N_0}_e$. 
\end{claim}
\begin{proof}
For every $n\geq N_0$ and objects $e \geq e_0$ we have a commutative diagram
$$
\xymatrix{
F(X^{N_0})_{e_0} \ar@{^{(}->}[r] \ar[d]_{\Psi_{e_o}^e} &
 F(X^{n})_{e_0} \ar@{^{(}->}[r] \ar[d]_{\Psi_{e_o}^e} &
  F(M)_{e_0} \ar[d]_{\Psi_{e_o}^e} \ar[dr]^{(\Psi_d^{e_o})^{-1}} & \\
F(X^{N_0})_{e} \ar@{^{(}->}[r] &
 F(X^{n})_{e} \ar@{^{(}->}[r] &
  F(M)_e \ar[r]^{(\Psi_d^e)^{-1}} &
  F(M)_{d}
}
$$
which by Corollary \ref{cor:injective_extension} all vertical extension maps are injective.

But we chose $F(X^{N_0})_{e_0}$ to be the subspace whose image inside $F(M)_{d}$ is maximal. It thus follows that all arrows in the above diagram are surjective. In particular, the injection $F(X^{N_0})_e \hookrightarrow F(X^n)_e$ is an isomorphism. Recalling the definition of $F$, we found that the inclusion
$$
\Hom_{G_e}(M_e, X^{N_0}_e) \hookrightarrow \Hom_{G_e}(M_e, X^{n}_e)
$$
is an isomorphism, where $X^{N_0}_e\subseteq N^n_e\subseteq M_e$ are $G_e$-subrepresentations. By Mashke's theorem, this happens precisely when $X^{N_0}_e = X^{n}_e$ thus proving the claim.
\end{proof}
It remains to show that we can find some $N_1\geq N_0$ such that for all objects $e < e_0$ the term $X^{N_1}_e$ stabilized. Indeed, since $\CCat$ is of $\FI$ type, there are only finitely many isomorphism classes of objects $\leq e_0$. Pick representatives for them $e_1,\ldots,e_n$ and consider the direct sum
$$
\oplus_{k=1}^n M_{e_k}.
$$
Since each $M_{e_k}$ is Noetherian (a finite dimensional vector space), this direct sum is Noetherian as well. We can therefore find $N_1 \geq N_0$ for which the sum
$$
\oplus_{k=1}^n X^{N_1}_{e_k} \subseteq \oplus_{k=1}^n M_{e_k}
$$
stabilized. Now for every $n \geq N_1$ and every object $e$ we have $X^{N_1}_e =X^n_e$ thus showing that $X^{N_1}_\bullet$ is a maximal element of our chain.
\end{proof}

\begin{remark}
	Making contact with related work, we remark that in \cite[Theorem 1.1]{GL-EI} Gan-Li list a set of combinatorial condition on categories of a certain type, which are sufficient for proving the Noetherian property \cite{GL-EI}. Their conditions are
	\begin{itemize} 
		\item \textbf{Surjectivity}: The groups $G_d$ act transitively on incoming morphisms $c\mor{}d$.
		\item \textbf{Bijectivity}: Some sequence of double-coset spaces $H_d \backslash G_d /H_d$ stabilizes as $d$ get sufficiently large.
	\end{itemize}
	These conditions are related to the present context as follows. First, the Surjectivity condition is incorporated into our definition of categories of $\FI$ type. As for Bijectivity, it was explained to me by Kevin Casto that by choosing a compatible system of morphisms $c\mor{}d$ for every pair $c\leq d$ one gets a natural isomorphism
	$$
	G_d \backslash \left(\Hom_\CCat(c,d)\times \Hom_\CCat (c,d)\right) \cong H_d \backslash G_d / H_d
	$$
	where $H_d \backslash G_d / H_d$ is the double-coset space the appears in the Bijectivity condition. In this sense, the objects considered in this work are a coordinate-free interpretation of those the appeared in \cite{GL-EI}. Arguing in this coordinate-free manner allows us to consider categories whose objects are not linearly ordered, avoid having to find a compatible system of morphisms, and show that the bijectivity condition holds for all categories of $\FI$ type. This is a direct result of Lemma \ref{lem:products_set_version}.
\end{remark}

The following stabilization result is a central motivation for one to be interested in the Noetherian property. It shows that finitely-generated $\CCat$-modules exhibit the same representation stability phenomena as free $\CCat$-module, only without the explicit stable range.
\begin{thm}[\textbf{Stabilization of finitely-generated $\CCat$-modules}]
If $M_\bullet$ is a finitely-generated $\CCat$-module, then all induced maps in the associated module of coinvariants are eventually isomorphisms. That is, there exists an upward-closed and cofinal set of objects $X$ such that if $c\in X$ and $d\geq c$, then the induced map 
$$
M_c/G_c \mor{} M_d/G_d
$$
is an isomorphism.

More generally, if $F_\bullet$ is any free $\CCat$-module, then the coinvariants of $F\otimes M$ eventually stabilizes in the above sense. In particular, the spaces $\Hom_{G_c}(F_c,M_c)$ stabilize as well.
\end{thm}
\begin{proof}
By the Noetherian property, a finitely-generated $\CCat$-module is finitely-presented, i.e. there exist free $\CCat$-modules $F^i_\bullet$ for $i=0,1$ and an exact sequence
$$
\xymatrix{
F^1 \ar[r] &
 F^0 \ar[r] &
  M \ar[r] &
  0
}
$$
Since the functor of coinvariants is right-exact we get a similar sequence of coinvariants. But by Theorem \ref{thm:coinvariant_stabilization} the coinvariants of a free $\CCat$-module stabilize in the desired sense. The Five-Lemma then implies that the same stabilization occurs for $M/G$.

For the more general statement, suppose $F_\bullet$ is some free $\CCat$-module. By the right-exactness of the tensor product it follows that
$$
\xymatrix{
F\otimes F^1 \ar[r] &
 F\otimes F^0 \ar[r] &
  F\otimes M \ar[r] &
  0
}
$$
is itself exact. Theorem \ref{thm:intro-free-modules}(1) shows that for $i=0,1$ the product $F\otimes F^i$ is free. Thus by the same reasoning as above stabilization follows.

Lastly, replacing $F$ with its dual $F^*$ (which is again free) and using the isomorphisms
$$
(F_c^*\otimes M_c )/G_c \cong \Hom(F_c,M_c)^{G_c} = \Hom_{G_c}(F_c,M_c)
$$
we find that the spaces on the right-hand side stabilize as well.
\end{proof}

\section{Example: Representation stability for $\FI^m$} \label{sec:rep_theory_of_FI}
This section is devoted to the category $\FI^m$, its free modules, and representation stability in this context. We also give an explicit description of $\FI^m$-character polynomials in terms of cycle-counting functions. The results presented below generalize to the category $(\FI_G)^m$ and its representation, where $G$ is some finite group, using the technique presented in \cite[Theorem 3.1.3]{SS-FIG}.

Recall that we denote the category of finite sets and injective functions by $\FI$. Consider the categorical power $\FI^m$, whose objects are ordered $m$-tuples of finite sets $\bar{n} = (n^{(1)},\ldots, n^{(m)})$, and whose morphisms $\bar{n}\mor{\bar{f}}\bar{n}'$ are ordered $m$-tuples of injections $\bar{f} = (f^{(1)},\ldots, f^{(m)})$ where $n^{(i)}\mor{f^{(i)}}n'^{(i)}$. In everything that follows we denote the $\FI^m$ analog of notions from $\FI$ by an over-line. The ordering on objects in $\FI^m$ is the following: $\bar{n}\leq \bar{n}'$ if and only if for every $1\leq i\leq m$ there is an inequality of sizes $|n^{(i)}|\leq |n'^{(i)}|$. The group of automorphisms of an object $\bar{n}$ is the product of symmetric groups $S_{n^{(1)}}\times\ldots\times S_{n^{(m)}}$, which we will denote by $S_{\bar{n}}$.

Many natural sequences of spaces and varieties are naturally parameterized by $\FI^m$. For example, fix some space $X$ and consider the following generalization of the configurations spaces
$$
\PConf^{(n_1,\ldots,n_m)}(X) := \{ [(x_1^{(1)},\ldots,x_{n_1}^{(1)}),\ldots,(x_1^{(m)},\ldots,x_{n_m}^{(m)})] \mid \forall i\neq j (x^{(i)}_{k_i}\neq x^{(j)}_{k_j}) \}
$$
inside the product $X^{n_1}\times \ldots \times X^{n_m}$. Every inclusion $\bar{n}\hookrightarrow \bar{n}'$ induces a continuous map by forgetting coordinates, so this is naturally a contravariant $\FI^m$-diagram of spaces. Applying a cohomology functor to this diagram of spaces yields an $\FI^m$-module. The special case of based rational maps $\Proj^1\mor{} \Proj^{m-1}$ was described in the introduction, to which theory below applies and gives Corollary \ref{thm:intro-homological-stability}.

\bigskip The category $\FI^m$ fits in with our general framework, as the following demonstrates.
\begin{prop}[\textbf{$\FI^m$ is of $\FI$ type}]
$\FI^m$ is a locally finite category of $\FI$ type. Pullbacks and weak push-outs are given by the corresponding operations in $\FI$ applied coordinatewise.
\end{prop}
\begin{proof}
First we consider the case $m=1$, i.e. we need to show that $\FI$ is indeed of $\FI$ type. $\FI$ is a subcategory of the category of finite sets, which has pullbacks and push-outs. The $\Set$-pullback of two injections itself has injective structure maps, and is thus naturally a pullback in $\FI$. Regarding weak push-outs, note that if
\squareDiagram{p}{c_1}{c_2}{d}{f_1}{f_2}{g_2}{g_2}
is a pullback diagram in $\Set$, then the images of $g_1$ and $g_2$ intersect precisely in the image of the composition $g_1\circ f_1 = g_2\circ f_2$. Thus if all four maps are injections, the universal function from the $\Set$ push-out $c_1\cup_p c_2$ into $d$ is injective. We therefore see that the $\Set$ push-out is a weak push-out in $\FI$. The other axioms of $\FI$ type are clear.

The case $m>1$ follows easily from the previous paragraph when pullbacks and weak push-outs are computed coordinatewise.
\end{proof}

We turn to the decomposition into irreducible subrepresentations. First we recall some of the terminology related to the case $m=1$.
\begin{definition}[\textbf{Padded partitions and irreducible representations}]
Recall that a \emph{partition} of a natural number $n$ is a sequence $\lambda = (\lambda_1 \geq \lambda_2 \geq \ldots \geq \lambda_k)$ such that $\sum_{i=1}^k \lambda_i = n$. In this case we write $\lambda \vdash n$ and refer to $n$ as the \emph{degree} of $\lambda$. This degree will be denoted by $|\lambda|$.

For every other natural number $d\geq |\lambda|+\lambda_1$ we define the \emph{padded partition}
\begin{equation}
\lambda(d) = (d-|\lambda|\geq \lambda_1 \ldots\geq \lambda_k) \vdash d
\end{equation}
By deleting the largest part of a partition, we see that every partition of $d$ is of the form $\lambda(d)$ for some partition $\lambda\vdash n < d$.

Recall that the partitions on $d$ are in one-to-one correspondence with the irreducible representations of $S_d$. Denote the corresponding irreducible representation by $V_{\lambda(d)}$.
\end{definition}

To consider the case $m>1$ recall that the irreducible representations of a product of finite groups $G\times H$ are given exactly by the pairs $V\boxtimes W$ where $V$ and $W$ are irreducible representations of $G$ and $H$ respectively. The symbol $\boxtimes$ is the usual tensor product on the underlying vector spaces $V$ and $W$ and the action of $G\times H$ on this product is defined by $(g,h).(v\otimes w) = g(v)\otimes h(w)$.
\begin{cor}
The irreducible representations of $S_{\bar{n}}=S_{n^{(1)}}\times\ldots\times S_{n^{(m)}}$ are precisely external tensor products of the form
$$
V_{\bar{\lambda}(\bar{n})} := V_{\lambda^{(1)}(n_1)} \boxtimes \ldots \boxtimes V_{\lambda^{(m)}(n_m)}
$$
where $|\lambda^{(i)}|+\lambda^{(i)}_1 \leq n^{(i)}$ for every $1\leq i\leq m$. Furthermore the character of such a product is given by the product of the individual characters.
\end{cor}
Following this observation we define a $\boxtimes$ operation on $\FI$-modules.
\begin{definition}[\textbf{External tensor product}]
Let $(M^{(1)},\ldots, M^{(m)})$ be an $m$-tuple of $\FI$-modules. We define their external tensor product to be the $\FI^m$-module
\begin{equation}
\bar{M} = M^{(1)}\boxtimes \ldots \boxtimes M^{(m)}
\end{equation}
by composing the functor $(M^{(1)},\ldots, M^{(m)}): \FI^m \mor{} (\Mod{R})^m$ with the $m$-fold tensor product functor on $R$-modules.
\end{definition}
We then see that if $\bar{n}$ is any object, then the $S_{\bar{n}}$-representation $\bar{M}_{\bar{n}}$ is precisely the external tensor product $M^{(1)}_{n^{(1)}}\boxtimes \ldots \boxtimes M^{(m)}_{n^{(m)}}$. Consequently, the character of $\bar{M}$ is the product of the $\FI$-characters of the factors.

\begin{remark}\label{rem:tensor_commutes_with_induction}
It is also important to note that the external tensor operation commutes with the $\Ind$ functors in the following sense:
\begin{equation}
\Ind_{\bar{n}}(V^{(1)}\boxtimes \ldots \boxtimes V^{(m)}) \cong \Ind_{n^{(1)}}(V^{(1)})\boxtimes \ldots \boxtimes \Ind_{n^{(m)}}(V^{(m)}).
\end{equation}
This can be verified e.g. by considering the definition of $\Ind$ in Definition \ref{def:Ind_functor}, and using the associativity and commutativity of the tensor product. 
\end{remark}

\begin{thm}[\textbf{Relating $\FI^m$-modules to $\FI$-modules}] \label{thm:FIm_to_FI}
The following relationships hold between the representation theory of $\FI^m$ and that of $\FI$.
\begin{enumerate}
\item Every free $\FI^m$ module of degree $\leq \bar{n}$ is the direct sum of external tensor products of free $\FI$-modules, where the $i$-th component is of degree $\leq n^{(i)}$.

\item Every $\FI^m$-character polynomial of degree $\leq \bar{n}$ decomposes as a sum of products of $\FI$-character polynomials, where the $i$-th factor has degree ${\leq n^{(i)}}$.
\end{enumerate}
\end{thm}

\begin{remark}
In most related work on the representation theory of the category $\FI$, free modules are called projective or $\FI\#$-modules. See \cite{CEF} for the relevant definitions and a proof that these concepts are equivalent.
\end{remark}

\begin{proof}
We start with the first assertion. Let $\bar{\lambda} = (\lambda^{(1)},\ldots,\lambda^{(m)})$ be an $m$-tuple of partitions and $\bar{n}$ some $m$-tuple of natural numbers satisfying $n^{(i)} \geq |\lambda^{(i)}|+\lambda^{(i)}_1$ for all $i=1,\ldots, m$. We apply the fact that $\Ind$ commutes with external tensor products to the irreducible $S_{\bar{n}}$-representation
$$
V_{\bar{\lambda}(\bar{n})} = V_{\lambda^{(1)}(n^{(1)})} \boxtimes \ldots \boxtimes V_{\lambda^{(m)}(n^{(m)})}.
$$
This gives a presentation
$$
\Ind_{\bar{n}}(V_{\bar{\lambda}}(\bar{n})) \cong \Ind_{n^{(1)}}(V_{\lambda^{(1)}(n^{(1)})})\boxtimes \ldots \boxtimes \Ind_{n^{(m)}}(V_{\lambda^{(m)}(n^{(m)})})
$$
which proves the first assertion of the theorem for $\Ind_{\bar{n}}(V)$ when $V$ is irreducible.

For a general $S_{\bar{n}}$-representation $V$, decompose $V$ into irreducible subrepresentations $V = V_1\oplus \ldots \oplus V_r$. Since $\Ind$ commutes with direct sums, the induction module $\Ind_{\bar{n}}(V)$ is a direct sum of external tensor products of induction $\FI$-modules.

Lastly, the assertion applies to all free $\FI^m$-modules, since they are directs sum of induction modules of the form previously considered.

\smallskip As for the second assertion, a character polynomials of degree $\leq \bar{n}$ is a $k$-linear combination of the characters of free $\FI^m$-modules of degree $\leq \bar{n}$. By the first statement such a free module is the sum of external tensor products of free $\FI$-modules with the appropriate bounds on their degrees. But the character of an external tensor product is the product of the individual characters, which in the case of products of free $\FI$-modules are by definition $\FI$-character polynomials. Thus every $\FI^m$-character polynomial is indeed a $k$-linear combination of products of $\FI$-character polynomials with the appropriate bound on degree.
\end{proof}

Theorem \ref{thm:FIm_to_FI} allows us to give an explicit description of the character polynomials of $\FI^m$ is terms of cycle counting functions.
\begin{definition}[\textbf{Cycle counting functions}]
For every natural number $k$, let $X_k: \coprod_{n} S_n \mor{} \N$ be the simultaneous class function on the symmetric groups $$X_k(\sigma)= \# \text{ of $k$-cycles appearing in $\sigma$}.$$

On the products $S_{n^{(1)}}\times \ldots \times S_{n^{(m)}}$ we define a similar function $X_k^{(i)}$ by $$X_k^{(i)}(\sigma^{(1)},\ldots, \sigma^{(m)})= \# \text{ of $k$-cycles appearing in $\sigma^{(i)}$}.$$
\end{definition}
The study of polynomials in the class functions $X_k$ dates back to Frobenius, and they are what is classically known as \emph{character polynomials}. The following proposition shows that our definition of character polynomials generalizes this classical idea.
\begin{thm}[\textbf{Character polynomials of $\FI^m$}]
The filtered $\C$-algebra of character polynomials of $\FI^m$ coincides with the polynomial ring
$$
R = \C[X_1^{(1)},\ldots, X_1^{(m)}, X_2^{(1)},\ldots,X_2^{(m)},\ldots].
$$
where we define $\deg(X_k^{(i)})=(0,\ldots,k,\ldots,0)=k \bar{e}^{(i)}$.
\end{thm}
\begin{proof}
We first prove this when $m=1$. For the inclusion $R\subseteq \Char_{\FI}$ we show that for every $k$ the function $X_k$ is indeed a character polynomial. Recall that in Example \ref{ex:FI_characters} we showed that for every cycle type $\mu = (\mu_1,\ldots,\mu_k)$ the associated character polynomial satisfies
\begin{equation}
\binom{X}{\mu}(\sigma) = \binom{X_1(\sigma)}{\mu_1}\ldots \binom{X_k(\sigma)}{\mu_k}.
\end{equation}
Thus by taking $\mu_k=1$ and $\mu_j=0$ for all $j\neq k$ we get a character polynomial
\begin{equation}
\binom{X}{\mu}(\sigma) = \binom{X_k(\sigma)}{1} = X_k(\sigma).
\end{equation}

\smallskip For the reverse inclusion, one can construct the right-hand side of
\begin{equation}
\binom{X}{\mu} = \binom{X_1}{\mu_1}\ldots \binom{X_k}{\mu_k}
\end{equation}
in the algebra generated by $X_1,X_2,\ldots$, thus realizing every generator $\binom{X}{\mu}$ of $\Char_{\FI}$. This concludes the proof in the case $m=1$.

For $m>1$, Theorem \ref{thm:FIm_to_FI} states that every $\FI^m$-character polynomial is a linear combination of external products of $\FI$-character polynomials. We saw that the latter class of functions is precisely the ring of polynomials in $X_1,X_2,\ldots$. The function $X_k^{(i)}$ is the external product of $X_k$ in the $i$-th coordinate with $1$'s in all other coordinates, and thus polynomials in $X_k^{(i)}$ clearly generate all linear combinations of external products of $X_1,X_2,\ldots$. This proves the claim.
\end{proof}

Our general theory of stabilization for inner products thus applies to expressions involving the functions $X_k^{(i)}$.
\begin{cor}[\textbf{Stabilization of inner products}] \label{cor:FI_product_stability}
The $S_{\bar{n}}$-inner product of two polynomials $P,Q\in \C[X^{(i)}_k: k\in \N,\, 1\leq i\leq m]$ does not depend on $\bar{n}$ for all $\bar{n}\geq \deg(P)+\deg(Q)$, where the degree of $X^{(i)}_k$ is $k \bar{e}^{(i)}$ and addition of degrees is defined coordinatewise.
\end{cor}
In the case $\CCat=\FI$ this result is proved in \cite[Theorem 3.9]{CEF-pointcounts} via a direct calculation of the $S_n$-inner products. The $\CCat=\FI_{\Z/2\Z}$-analog is proved in \cite{Wi}. When $G$ is any other finite group, a non-effective analog of Corollary \ref{cor:FI_product_stability} for $\CCat=\FI_G$ is implicit in \cite[Theorem 3.2.2]{SS-FIG}.

\bigskip We turn to discussing representation stability for $\FI^m$. First consider the case $m=1$: the irreducible representations of symmetric groups of different orders are naturally related in the following sense.
\begin{fact}[\textbf{The modules $V_{\lambda(\bullet)}$}]
For every partition $\lambda \vdash |\lambda|$, there exists an $\FI$-submodule of $\Ind_{|\lambda|}(V_{\lambda})$, which we will denote by $V_{\lambda(\bullet)}$, whose value at every $d\geq |\lambda|+\lambda_1$ is isomorphic to the irreducible $S_d$-representation $V_{\lambda(d)}$. Moreover, for every partition $\lambda$ there exists a character polynomial $P_\lambda$ of degree $|\lambda|$ such that the character of $V_{\lambda(\bullet)}$ coincides with $P_\lambda$ on $S_d$ for all $d \geq |\lambda|+\lambda_1$.

See \cite{CEF} for the existence of $V_{\lambda(\bullet)}$ and \cite[Example I.7.14]{Mac} for $P_{\lambda}$.
\end{fact}

This fact extends to all $m>1$ via the external tensor product.
\begin{cor}[\textbf{The modules $V_{\bar{\lambda}(\bullet)}$}]
For every $m$-tuple of partitions $\bar{\lambda} = (\lambda^{(1)},\ldots,\lambda^{(m)})$ there exists an $\FI^m$-submodule of $\Ind_{|\bar{\lambda}|}(V_{\bar{\lambda}})$, which we will denote by $V_{\bar{\lambda}(\bullet)}$, whose value at $\bar{n}$ is the $S_{\bar{n}}$-irreducible representation 
$$
V_{\bar{\lambda}(\bar{n})}:=V_{\lambda^{(1)}(n^{(1)})} \boxtimes \ldots \boxtimes V_{\lambda^{(m)}(n^{(m)})}
$$
for all $\bar{n}\geq |\bar{\lambda}|+\bar{\lambda}_1$. Here $|\bar{\lambda}|$ is the $m$-tuple $(|\lambda^{(1)}|,\ldots, |\lambda^{(m)}|)$, the expression $\bar{\lambda}_1$ is $(\lambda^{(1)}_1,\ldots,\lambda^{(m)}_1)$ and $+$ coincides with coordinatewise addition.

Moreover, the character of $V_{\bar{\lambda}(\bullet)}$ coincides with the character polynomial ${P_{\bar{\lambda}}:=P_{\lambda^{(1)}}\cdot \ldots \cdot P_{\lambda^{(m)}}}$ of degree $|\bar{\lambda}|$.
\end{cor}

These sequences of irreducible representations allow us to formulate the notion of representation stability for free $\FI^m$-modules.
\begin{thm}[\textbf{Representation stability for $\FI^m$}] \label{thm:rep_stability}
Suppose $F_\bullet$ is a free $\FI^m$-module that is finitely-generated in degree $\leq \bar{n}$. Then there exist $m$-tuples of partitions $\bar{\lambda}_1,\ldots, \bar{\lambda}_k$, satisfying $|\bar{\lambda}_j|\leq \bar{n}$ for all $j=1,\ldots,k$, such that for all $\bar{d}\geq 2\times \bar{n} = (2n^{(1)},\ldots,2n^{(m)})$ the $S_{\bar{d}}$-module $F_{\bar{d}}$ decomposes into irreducibles as
$$
F_{\bar{d}} \cong (V_{\bar{\lambda}_1(\bar{d})})^{r_1}\oplus \ldots \oplus (V_{\bar{\lambda}_k(\bar{d})})^{r_k}
$$
and the multiplicities $r_1,\ldots,r_k$ do not depend on $\bar{d}$.
\end{thm}
\begin{note}
The original definition of Representation Stability given in \cite{CF} includes additional injectivity and surjectivity conditions on top of the stabilization of irreducible decompositions. We will not discuss these aspects of the definition, although the reader familiar with them will readily notice that they are immediately satisfied by all free $\CCat$-modules.
\end{note}

\begin{proof} [Proof of Theorem \ref{thm:rep_stability}]
The case $m=1$ asserts the representation stability of free finitely-generated $\FI$-modules. This follows directly from the Branching rule for inducing representations of the symmetric group (see \cite{FH}), and is proved in \cite[Theorem 1.13]{CEF}.

For $m>1$ the statement follows from Theorem \ref{thm:FIm_to_FI} using the corresponding statement in the case $m=1$. Since every free $\FI^m$-module $M_\bullet$ is a sum of the external tensor products of free $\FI$-modules, and each of those decomposes as a stabilizing direct sum of irreducibles, the same is true for $M_\bullet$.
\end{proof}

At the level of character polynomials Theorem \ref{thm:rep_stability} translates to the following orthonormality statement.
\begin{cor}[\textbf{Spectral orthonormal basis for character polynomials}] \label{cor:orthonormal_basis}
The character polynomials
$$
\left\{ P_{\bar{\lambda}}:=P_{\lambda_1}\cdot \ldots \cdot P_{\lambda_m} \right\}_{|\bar{\lambda}|\leq \bar{n}}
$$
form an orthonormal basis for all $\FI^m$-character polynomials of degree $\leq \bar{n}$ with respect to the inner product
\begin{equation}
\langle P, Q\rangle = \lim_{\bullet\rightarrow \infty} \langle P, Q\rangle_{\bullet} = \langle P, Q\rangle_{\deg(P)+\deg(Q)}.
\end{equation}
\end{cor}

\end{document}